\documentclass[letterpaper]{article}
\usepackage[letterpaper, margin=1in]{geometry}

\usepackage{mathtools}
\usepackage{graphicx,color} %
\usepackage{subfigure} 
\usepackage{cite}
\usepackage{algorithm}
\usepackage{amsthm}

\usepackage{algorithmic}
\usepackage{hyperref}
\usepackage{amsfonts}
\usepackage{amsmath}
\usepackage[normalem]{ulem}
\usepackage{multicol}

\newtheorem{theorem}{Theorem}
\newtheorem{lemma}[theorem]{Lemma}

\newcommand{\citep}{\cite}

\DeclarePairedDelimiter\floor{\lfloor}{\rfloor}

\def\mbf{\mathbf}
\def\mbs{\boldsymbol}
\def\mc{\mathcal}
\def\spec{\sigma_1}
\def\lmax{\lambda_{\mathrm{max}}}

\DeclareMathOperator*{\argmin}{argmin}

\DeclareMathOperator{\Tr}{Tr}

\newcommand{\eqdef}{\stackrel{\text{def}}{=}}
\newcommand{\R}{\mathbb{R}}                      %
\newcommand{\Prob}{\mathbb{P}}                   %
\newcommand{\E}{\mathbb{E}}                      %

\newcommand{\Y}{\mathbf{Y}}

\newcommand{\calF}{\mathcal{F}}

\newcommand{\x}{ {\bf x}}

\newcommand{\X}{ {\bf X}}
\newcommand{\Z}{ {\bf Z}}
\newcommand{\Q}{ {\bf Q}}

\newcommand{\w}{  {\bf w}}

\newcommand{\removed}[1]{}
\newcommand{\norm}[1]{\left\| #1 \right\|}

\newcommand{\trans}{{\top}}

\newcommand{\bP}{\mathbf{P}}

\newcommand{\g}{\mathbf{g}}

\newcommand{\sgi}[2]{\beta_{#1,#2}}
\newcommand{\sgt}[1]{\gamma_{#1}}
\newcommand{\SGI}[1]{\mbs{\beta}_t}
\newcommand{\SGT}{\mbs{\gamma}}

\newcommand{\dists}{\zeta}
\newcommand{\specnorm}{\rho}
\newcommand{\commfrac}{\nu}
\newcommand{\reg}{\mu}
\newcommand{\step}{\eta}

\newcommand{\rcv}{\texttt{RCV1}}
\newcommand{\ctype}{\texttt{Covertype}}

\begin{document}

\title{Data Dependent Convergence For Consensus Stochastic Optimization}

\author{Avleen S.~Bijral,
        Anand D.~Sarwate,
        ~Nathan~Srebro%
}

\maketitle

\begin{abstract}
We study a distributed consensus-based stochastic gradient descent (SGD) algorithm and show that the rate of convergence involves the spectral properties of two matrices: the standard spectral gap of a weight matrix from the network topology and a new term depending on the spectral norm of the sample covariance matrix of the data. This data-dependent convergence rate shows that distributed SGD algorithms perform better on datasets with small spectral norm. Our analysis method also allows us to find data-dependent convergence rates as we limit the amount of communication. Spreading a fixed amount of data across more nodes slows convergence; for asymptotically growing data sets we show that adding more machines can help when minimizing twice-differentiable losses.
\end{abstract}

\section{Introduction}
Decentralized optimization algorithms for statistical computation and machine learning on large data sets try to trade off efficiency (in terms of estimation error) and speed (from parallelization).
From an empirical perspective, it is often unclear when these methods will work for a particular data set, and to what degree additional communication can improve performance. For example, in high-dimensional problems communication can be costly. We would therefore like to know when limiting communication is feasible or beneficial. The theoretical analysis of distributed optimization methods has focused on  providing strong \textit{data-independent} convergence rates under 
analytic assumptions on the objective function such as convexity and smoothness. 
In this paper we show how the tradeoff between efficiency and speed is  affected by the data distribution itself. We study a class of distributed optimization algorithms and prove an upper bound on the error that depends on the spectral norm of the data covariance. By tuning the frequency with which nodes communicate, we obtain a bound that depends on data distribution, network size and topology, and amount of communication. This allows us to interpolate between regimes where communication is cheap (e.g. shared memory systems) and those where it is not (clusters and sensor networks).

We study the problem of minimizing a regularized convex function~\citep{RakhShamir:12arxiv} of the form
\begin{align}
 J(\w) &=  \sum_{i=1}^N\frac{\ell(\w^{\trans}\x_i ; y_i)}{N} + \frac{\reg}{2} \norm{ \w }^2  \label{eq:optForm} \\
       &= \E_{\x \sim \hat{\mc{P}}}\left[\ell(\w^{\trans}\x ; y)\right] + \frac{\reg}{2} \norm{ \w }^2 \notag
 \end{align}
where $\ell(\cdot)$ is convex and Lipschitz and the expectation is with respect to the empirical distribution $\hat{\mc{P}}$ corresponding to a given data set with $N$ total data points $\{ (\mbf{x}_i, y_i) \}$. We will assume $\x_i \in \mathbb{R}^d$ and $y_i \in \mathbb{R}$. This \textit{regularized empirical risk minimization} formulation encompasses algorithms such as support vector machine classification, ridge regression, logistic regression, and others~\cite{UMLbook}. For example $\x$ could represent $d$ pixels in a grayscale image and $y$ a binary label indicating whether the image is of a face: $\w^{\trans}\x$ gives a confidence value about whether the image is of a face or not.
We would like to solve such problems using a network of $m$ processors connected via a network (represented by a graph indicating which nodes can communicate with each other). The system would distribute these $N$ points across the $m$ nodes, inducing local objective functions $J_j(\w)$ approximating \eqref{eq:optForm}. 

In such a computational model, nodes can perform local computations and send messages to each other to jointly minimize \eqref{eq:optForm}. The strategy we analyze is what is referred to as distributed primal averaging~\citep{nedicDistributedOptimization}: each node in the network processes points sequentially, performing a SGD update locally and averaging the current iterate values of their neighbors after each gradient step. This can also be thought of as a distributed consensus-based version of Pegasos \citep{SSSC11:pegasos} when the loss function is the hinge loss. %
We consider a general topology with $m$ nodes attempting to minimize a global objective function $J(\w)$ that decomposes into a sum of $m$ local objectives: $J(\w) = \sum_{i=1}^{m} J_i(\w)$.  This is a model for optimization in systems such as data centers, distributed control systems, and sensor networks. %

\textbf{Main Results.} Our goal in this paper is to characterize how the spectral norm $\specnorm^2 = \spec( \E_{\hat{\mc{P}}}[ \x \x^{\trans} ])$ of the sample covariance affects the rate of convergence of stochastic consensus schemes under different communication requirements. Elucidating this dependence can help guide empirical practice by providing insight into when these methods will work well. We prove an upper bound on the suboptimality gap for distributed primal averaging that depends on $\specnorm^2$ as well as the mixing time of the weight matrix associated to the algorithm. 
 Our result shows that networks of size $m<\frac{1}{\specnorm^2}$ gain from parallelization. To understand the communication-limited regime, we extend our analysis to intermittent communication. In a setting with finite data and sparse connectivity, convergence will deteriorate with increasing $m$ because we split the data to more machines that are farther apart. We also show that by using a mini-batching strategy we can offset the penalty of infrequent communication by communicating after a mini-batch (sub)gradient step. Finally, in an asymptotic regime with infinite data at every node we show using results of Bianchi et al.~\cite{BianchiFortHachem:13IEEETrans} that for twice-differentiable loss functions this network effect disappears and that we gain from additional parallelization.

\textbf{Related Work.} Several authors have proposed distributed algorithms  involving nodes computing local gradient steps and averaging iterates, gradients, or other functions of their neighbors~\citep{nedicDistributedOptimization,dualAveraging,distrStochSubgrOpt}. By alternating local updates and consensus with neighbors, estimates at the nodes converge to the optimizer of $J(\cdot)$. 
In these works no assumption is made on the local objective functions and they can be arbitrary. Consequently the convergence guarantees do not reflect the setting when the data is homogenous (for e.g. when data has the same distribution), specifically error increases as we add more machines. This is counterintuitive, especially in the large scale regime, since this suggests that despite homogeneity the methods perform worse than the centralized setting (all data on one node). 

We provide a first data-dependent analysis of a consensus based stochastic gradient method in the homogenous setting  and demonstrate that there exist regimes %
where we benefit from having more machines in any network. \removed{To mitigate the effect of limited communication, we propose and analyze a mini-batched extension to reduce communication costs. We interpret this as an intermediate regime between full communication and one-shot communication~\citep{ZhangDW:12,ShamirSrebroZhang:14icml}. Finally, we show that for twice-differentiable losses, having more machines always helps (via a variance reduction) in the infinite data regime, using results of Bianchi et al.~\cite{BianchiFortHachem:13IEEETrans}.}

In contrast to our stochastic gradient based results, data dependence via the Hessian of the objective has also been demonstrated in parallel coordinate descent based approaches of Liu et al.~\cite{LiuWrightReBittSri} and the Shotgun algorithm of Bradley et al.~\cite{BradleyKyrola}. The assumptions differ from us in that the objective function is assumed to be smooth~\cite{LiuWrightReBittSri} or $\mathcal{L}_1$ regularized~\cite{BradleyKyrola}. Most importantly, our results hold for arbitrary networks of compute nodes, while the coordinate descent based results hold only for networks where all nodes communicate with a central aggregator (sometimes referred to as a master-slave architecture, or a star network), which can be used to model shared-memory systems.
Another interesting line of work is the impact of delay on convergence in distributed optimization~\cite{AgarwalD:11nips}. These results show that delays in the gradient computation for a star network are asymptotically negligible when optimizing smooth loss functions. We study general network topologies but with intermittent, rather than delayed communication. Our result suggest that certain datasets are more tolerant of skipped communication rounds, based on the spectral norm of their covariance.

We take an approach similar to that of Tak\'{a}\v{c} et al.~\citep{TakacBRS:13icml} who developed a spectral-norm based analysis of mini-batching for non-smooth functions. We decompose the iterate in terms of the data points encountered in the sample path~\citep{CotterSSS:11nips}. This differs from analysis based on smoothness considerations alone~\citep{CotterSSS:11nips,AgarwalD:11nips,DekelGSX:12mini,ShamirSrebroZhang:14icml} and gives practical insight into how communication (full or intermittent) impacts the performance of these algorithms. Note that our work is fundamentally different in that these other works either assume a centralized setting \citep{CotterSSS:11nips,DekelGSX:12mini,ShamirSrebroZhang:14icml} or implicitly assume a specific network topology (e.g. \citep{ZhangDW:12} uses a star topology). For the main results we only assume strong convexity while the existing guarantees for the cited methods depend on a variety of regularity and smoothness conditions.

 \textbf{Limitation.} In the stochastic convex optimization (see for e.g. \cite{StochConOpt}) setting the quantity of interest is the population objective corresponding to problem \ref{eq:optForm}. When minimizing this population objective our results suggest that adding more machines worsens convergence (See Theorem \ref{theorem:mainThrm}). For finite data our convergence results satisfy the intuition that adding more nodes in an arbitrary network will hurt convergence. The finite homogenous setting is most relevant in settings such as data centers, where the processors hold data which essentially looks the same. %In some sensor network applications such as global calibration, this assumption may also hold, but for spatially varying phenomena the task may be different and our model may no longer be as relevant~\cite{ForestFireDetect}.
In the infinite or large scale data setting, common in machine learning applications, this is  counterintuitive since when each node has infinite data, any distributed scheme including one on arbitrary networks shouldn't perform worse than the centralized scheme (all data on one node). Thus our analysis is limited in that it doesn't unify the stochastic optimization and the consensus setting in a completely satisfactory manner. To partially remedy this we explore consensus SGD for smooth strongly convex objectives in the asymptotic regime and show that one can gain from adding more machines in any network.

In this paper we focus on a simple and well-studied protocol~\cite{nedicDistributedOptimization}. However, our analysis approach and insights may yield data-dependent bounds for other more complex algorithms such as distributed dual averaging~\cite{dualAveraging}. 
More sophisticated gradient averaging schemes such as that of Mokhtari and Ribeiro~\cite{MokhtariR:16jmlr} can exploit dependence across iterations~\cite{ShiLWY:15extra,SchmidLF:2015} to improve the convergence rate; analyzing the impact of the data distribution is considerably more complex in these algorithms.

We believe that our results provide a first step towards understanding data-dependent bounds for distributed stochastic optimization in settings common to machine learning. Our analysis coincides with phenomenon seen in practice: for data sets with small $\specnorm$, distributing the computation across many machines is beneficial, but for data with larger $\specnorm$ more machines is not necessarily better. \removed{We provide upper bounds on the gap between the iterates and the optimal solution: these bounds do not immediately yield parameters for practical use, but}Our work suggests that taking into account the data dependence can improve the empirical performance of these methods.

\section{Model}
We will use boldface for vectors.  Let $[k] = \{1,2,\ldots, k\}$.  Unless otherwise specified, the norm $\norm{\cdot}$ is the standard Euclidean norm. The spectral norm of a matrix $A$ is defined to be the largest singular value $\spec(A)$ of the matrix $A$ or equivalently the square root of the largest eigenvalue of $A^{\top} A$. For a graph $\mc{G} = (\mc{V}, \mc{E})$ with vertex set $\mc{V}$ and edge set $\mc{E}$, we will denote the neighbors of a vertex $i \in \mc{V}$ by $\mc{N}(i) \subseteq \mc{V}$.

\textbf{Data model.} Let $\mc{P}$ be a distribution on $\mathbb{R}^{d+1}$ such that 
for $(\x,y) \sim \mc{P}$, we have $\norm{\x} \le 1$ almost surely.  Let $S = \{\x_1, \x_2, \ldots, \x_N\}$ be i.i.d sample of $d$-dimensional vectors from $\mc{P}$ and let $\hat{\mc{P}}$ be the empirical distribution of $S$. Let $\mbs{\hat{\Sigma}} = \E_{\x \sim \hat{\mc{P}}}[ \x \x^{\trans} ]$ be the sample second-moment matrix of $S$. Our goal is to express the performance of our algorithms in terms of $\specnorm^2 = \spec(\mbs{\hat{\Sigma}})$, the spectral norm of $\mbs{\hat{\Sigma}}$. The spectral norm $\specnorm^2$ can vary significantly across different data sets. For example, for sparse data sets $\specnorm^2$ is often small. This can also happen if the data happens to lie in low-dimensional subspace (smaller than the ambient dimension $d$).

\textbf{Problem.}  Our problem is to minimize a particular instance of \eqref{eq:optForm} where the expectation is over a finite collection of data points:
	\begin{align}
	\w^* \eqdef \argmin_{\w} J(\w)
	\label{eq:globalopt}
	\end{align}
\removed{We consider a model in which $m$ individual nodes\removed{ or processors} operate in a discrete-time fashion to compute $\w^*$. In each time-slot (or iteration) they can perform local computations and pass messages for a fixed number $T$ of iterations.} Let $\hat{\w}_j(t)$ be the estimate of $\w^*$ at node $j \in [m]$ in the $t$-th iteration. We bound the expected gap (over the data distribution) at iteration $T$ between $J(\w^*)$ and the value $J(\hat{\w}_i(T))$ of the global objective $J(\hat{\w}_j(T))$ at the output $\hat{\w}_j(T)$ of each node $j$ in our distributed network. \removed{It would also be interesting to study the gap between the iterates and the minimizer of the population objective in \eqref{eq:optForm}; we defer this for future work.}
We will denote the subgradient set of $J(\w)$ by $\partial J(\w)$ and a subgradient of $J(\w)$ by $\nabla J(\w) \in \partial J(\w)$.

In our analysis we will make the following assumptions about the individual functions $\ell(\w^{\trans}\x)$: (a) The loss functions $\{ \ell(\cdot) \}$ are convex, and (b) The loss functions $\{ \ell(\cdot; y) \}$ are $L$-Lipschitz for some $L > 0$ and all $y$. Note that $J(\w)$ is $\reg$-strongly convex due to the $\ell_2$-regularization. Our analysis will not depend on the the response $y$ except through the Lipschitz bound $L$ so we will omit the explicit dependence on $y$ to simplify the notation in the future.

\textbf{Network Model.} We consider a model in which minimization in \eqref{eq:globalopt} must be carried out by $m$ \removed{computational devices, or }nodes.  These nodes are arranged in a network whose topology is given by a graph $\mc{G}$ -- an edge $(i,j)$ in the graph means nodes $i$ and $j$ can communicate.  A matrix $\bP$ is called graph conformant if $P_{ij} > 0$ only if the edge $(i,j)$ is in the graph.  We will consider algorithms which use a doubly stochastic and graph conformant sequence of matrices $\bP(t)$.

\textbf{Sampling Model.} We assume the $N$ data points are divided evenly and uniformly at random among the $m$ nodes, and define $n \eqdef N/m$ to be the number of points at each node. This is a necessary assumption since our bounds are data dependent and depend on subsampling bounds of spectral norm of certain random submatrices. However our data independent bound holds for arbitrary splits.   Let $S_i$ be the subset of $n$ points at node $i$.  The local stochastic gradient procedure consists of each node $i \in [m]$ sampling from $S_i$ with replacement.  This is an approximation to the local objective function
\begin{align}
J_i(\w) = \sum_{j \in S_i} \frac{\ell(\w^{\trans} \x_{i,j})}{n} + \frac{\reg}{2} \norm{\w}^2.
\end{align}

\textbf{Algorithm.} In the subsequent sections we analyze the distributed version (Algorithm \ref{alg:DiSCO}) of standard SGD. This algorithm is not new~\citep{nedicDistributedOptimization,distrStochSubgrOpt} and has been analyzed extensively in the literature. The step-size $\eta_t=1/(\mu t)$ is commonly used for large scale strongly convex machine learning problems like SVMs (e.g.-\cite{SSSC11:pegasos}) and ridge regression: to avoid an extra parameter in the bounds, we take this setting. \removed{More discussion on the optimality of the step-size for general strongly convex functions can be found in \cite{RakhShamir:12}.}
In Algorithm \ref{alg:DiSCO} node $i$ samples a point uniformly with replacement from a local pool of $n$ points and then updates its iterate by computing a weighted sum with its neighbors followed by a local subgradient step. The selection is uniform to guarantee that the subgradient is an unbiased estimate of a true subgradient of the local objective $J_i(\w)$, and greatly simplifies the analysis. Different choices of $\bP(t)$ will allow us to understand the effect of limiting communication in this distributed optimization algorithm. 

\begin{algorithm}[!htb]
   \caption{Consensus Strongly Convex Optimization}
   \label{alg:DiSCO}
\begin{algorithmic}
   \STATE {\bfseries Input:} $\{\x_{i,j}\}$,where $i\in [m]$ and $j \in [n]$ and $N=mn$, matrix sequence $\bP(t)$, $\reg> 0$, $T \geq 1$
   \STATE
   \STATE \COMMENT{Each $i \in [m]$ executes}
   \STATE {\bfseries Initialize:} set $\w_i(1) = {\bf 0} \in \R^d$.
   \FOR{$t=1$ {\bfseries to} $T$}
   \STATE Sample $\x_{i,t}$ uniformly with replacement from $S_i$.
   \STATE Compute $\g_i(t) \in  \partial\ell(\w_i(t)^{\trans}\x_{i,t})\x_{i,t} + \reg \w_i(t)$
   \STATE $\w_i(t+1) = \sum_{j=1}^m \w_j(t) P_{ij}(t) - \step_t \g_i(t)\label{eq:updRule}$
   \ENDFOR
    \STATE {\bfseries Output:} 
    $\hat{\w}_i(T)=\frac{1}{T}\sum_{t=1}^{T} \w_i(t)\label{eq:avgPred}$ for any $i \in [m]$.
\end{algorithmic}
\end{algorithm}

\textbf{Expectations and probabilities.} There are two sources of stochasticity in our model: the first in the split of data points to the individual nodes, and the second in sampling the points during the gradient descent procedure. We assume that the split is done uniformly at random, which implies that the expected covariance matrix at each node is the same as the population covariance matrix $\mbs{\hat{\Sigma}}$. Conditioned on the split, we assume that the sampling at each node is uniformly at random from the data point at that node, which makes the stochastic subgradient an unbiased estimate of the subgradient of the local objective function. Let  $\mc{F}_t$ be the sigma algebra generated by the random point selections of the algorithm up to time $t$, so that the iterates $\{\w_i(t) : i \in [m]\}$ are measurable with respect to $\mc{F}_t$.

\section{Convergence and Implications~\label{section:CommEveryIteration}}

Methods like Algorithm \ref{alg:DiSCO}, also referred to as primal averaging, have been analyzed previously~\citep{nedicDistributedOptimization,distrStochSubgrOpt,DistStronglyConvex}. In these works it is shown that the convergence properties depend on the structure of the underlying network via the second largest eigenvalue of $\bP$.  We consider in this section the case when $\bP(t)=\bP$ for all $t$ where $\bP$ is a fixed Markov matrix.  This corresponds to a synchronous setting where communication occurs at every iteration.

We analyze the use of the step-size $\step_t=1/(\reg t)$ in Algorithm \ref{alg:DiSCO} and show that the convergence depends on the spectral norm $\specnorm^2=\sigma_1(\mbf{\hat{\Sigma}})$ of the sample covariance matrix.

\begin{theorem}
Fix a Markov matrix $\bP$ and let $\specnorm^2 =\sigma_1(\mbf{\hat{\Sigma}})$ denote the spectral norm of the covariance matrix of the data distribution. Consider Algorithm \ref{alg:DiSCO} when the objective $J(\w)$ is strongly convex, $\bP(t) = \bP$ for all $t$, and $\step_t=1/(\reg t)$.  Let $\lambda_2(\bP)$ denote the second largest eigenvalue of $\bP$.  Then if the number of samples on each machine $n$ satisfies 
	\begin{align}
	n > \frac{4}{3 \rho^2} \log \left( d\right)
	\end{align}
and the number of iterations $T$ satisfies
	\begin{align}
	T &> 2 e  \log(1/\sqrt{ \lambda_2(\bP)}) \\
	\frac{T}{\log(T)} &> \max\left(\frac{4}{3 \rho^2} \log \left( d\right), \frac{ \left( \frac{8}{5} \right)^{\frac{1}{4}} \sqrt{ m/\specnorm }}{\log(1/\lambda_2(\bP))}\right), \label{eq:thm_Tbound}
	\end{align}
then the expected error for each node $i$ satisfies
\begin{align}
&\E\left[ J(\hat{\w}_i(T)) -J(\w^{*})  \right] \ \le \ \notag \\
&\left(\frac{1}{m} + \frac{100\sqrt{m\specnorm^2}\cdot \log T}{1-\sqrt{\lambda_2(\bP)}} \right)\cdot \frac{L^2}{\reg} \cdot \frac{\log T}{T}.
\label{eq:morecomm:totalerr}
\end{align}
\label{theorem:mainThrm}
\end{theorem}

\noindent \textit{Remark 1:} Theorem \ref{theorem:mainThrm} indicates that the number of machines should be chosen as a function of $\specnorm$. We can identify three sub-cases of interest:

\textbf{Case (a):} \textit{$m\le \frac{1}{\specnorm^{2/3}}$}: In this regime since $1/m > \sqrt{m\specnorm^2}$ (ignoring the constants and the $\log T$ term) we always benefit from adding more machines.

\textbf{Case (b):} \textit{$\frac{1}{\specnorm^{2/3}} < m\le \frac{1}{\specnorm^2}$}: 
The result tells us that there is no degradation in the error and the bound improves by a factor $\sqrt{m}\specnorm$. Sparse data sets generally have a smaller value of $\specnorm^2$ (as seen in Tak\'{a}\v{c} et al.~\cite{TakacBRS:13icml}); Theorem \ref{theorem:mainThrm} suggests that for such data sets we can use a larger number of machines without losing performance. However the requirements on the number of iterations also increases. This provides additional perspective on the observation by Tak\'{a}\v{c} et al~\cite{TakacBRS:13icml} that sparse datasets are more amenable to parallelization via mini-batching. The same holds for our type of parallelization as well.

\textbf{Case (c):} \textit{$m > \frac{1}{\specnorm^2}$}: 
In this case we pay a penalty $\sqrt{m\specnorm^2} \ge 1$ suggesting that for datasets with large $\specnorm$ we should expect to lose performance even with relatively fewer machines.

Note that $m>1$ is implicit in the condition $T > 2 e  \log(1/\sqrt{ \lambda_2)}) $ since $\lambda_2 =0$ for $m=1$. This excludes the single node Pegasos~\citep{TakacBRS:13icml} case. Additionally in the case of general strongly convex losses (not necessarily dependent on $\w^{\trans}\x$) we can obtain a convergence rate of $\mc{O}(\log^2(T)/T)$. We do not provide the proof here.

\removed{\noindent \textit{Remark 2:} The lower bound on the number of iterations \ref{eq:thm_Tbound} can be considerably improved by instead looking at the intrinsic dimension of the data since for several real datasets the intrinsic dimension can be much smaller than the dimension of the ambient space. However this requires us to assume a lower bound on the norm of the data samples, which is a less natural assumption.}

\section{Stochastic Communication~\label{section:StochComm}}

In this section we generalize our analysis in Theorem \ref{theorem:mainThrm} to handle time-varying and stochastic communication matrices $\bP(t)$.   In particular, we study the case where the matrices are chosen i.i.d.~over time.  %
Any strategy that doesn't involve communicating at every step will incur a larger gap between the local node estimates and their average. We call this the network error.  Our goal is to show how knowing $\specnorm^2$ can help us balance the network error and optimality gap.

First we bound the network error for the case of stochastic time varying communication matrices $P(t)$ and then a simple extension leads to a generalized version of Theorem \ref{theorem:mainThrm}.

\begin{lemma}\label{lemma:avgdevBndStoch}
Let $\{\bP(t)\}$ be a i.i.d sequence of doubly stochastic Markov matrices and consider Algorithm \ref{alg:DiSCO} when the objective $J(\w)$ is strongly convex. 
We have the following inequality for the expected squared error between the iterate $\w_i(t)$ at node $i$ at time $t$ and the average $\bar{\w}(t)$ defined in Algorithm \ref{alg:DiSCO}:
	\begin{align}
	\sqrt{\E\left[\norm{ \bar{\w}(t)-\w_i(t) }^2 \right]}
		&\le \frac{2L}{\reg} \cdot \frac{\sqrt{m}}{b} \cdot \frac{\log(2bet^2)}{t}.
	\end{align}
where $b=\log\left(1/\lambda_2\left(\E\left[\bP^2(t)\right]\right)\right)$.
\end{lemma}

Armed with Lemma \ref{lemma:avgdevBndStoch} we prove the following theorem for Algorithm \ref{alg:DiSCO} in the case of stochastic communication.
\begin{theorem}
Let $\{\bP(t)\}$ be an i.i.d sequence of doubly stochastic matrices and $\specnorm^2=\sigma_1(\hat{\mbf{\Sigma}})$ denote the spectral norm of the sample covariance matrix. Consider Algorithm \ref{alg:DiSCO} when the objective $J(\w)$ is strongly convex, and $\step_t=1/(\reg t)$. Then if the number of samples on each machine $n$ satisfies 
	\begin{align}
	n > \frac{4}{3\specnorm^2} \log \left( d\right) 
	\end{align}
and the number of iterations $T$ satisfies
	\begin{align}
	T &> 2 e  \log(1/\sqrt{ \lambda_2(\E\left[\bP^2(t)\right])}) 
	\end{align}
and
	\begin{align}
	\frac{T}{\log(T)} &>
	 \max\Bigg( \frac{4}{3\specnorm^2}\log(d),
	 	\notag \\
		&\hspace{0.8in}
		\sqrt{\frac{8}{5}} \cdot\sqrt{ \frac{m}{\specnorm^2} } \cdot \frac{1}{\log(1/\lambda_2(\E\left[\bP^2(t)\right]))} \Bigg), 
	\end{align}
then the expected error for the output of each node $i$ satisfies 
\begin{align}
&\E\left[ J(\hat{\w}_i(T)) -J(\w^{*})  \right]  
\notag \\
&\le \left(\frac{1}{m} + \frac{100\sqrt{m\specnorm^2}\cdot \log T}{1-\sqrt{\lambda_2(\E\left[\bP^2(t)\right])}} \right)\cdot \frac{L^2}{\mu} \cdot \frac{\log T}{T}.
\end{align}

\label{theorem:mainThrmStoch}
\end{theorem}
\noindent \textit{Remark:} This result generalizes the conclusions of Theorem \ref{theorem:mainThrm} to the case of stochastic communication schemes. Thus allowing for the data dependent interpretations of convergence in a more general setting.

\section{Limiting Communication~\label{section:LessComm}}

As an application of the stochastic communication scenario of Theorem \eqref{theorem:mainThrmStoch} we now analyze the effect of reducing the communication overhead of Algorithm \ref{alg:DiSCO}. This reduction can improve the overall running time (``wall time'') of the algorithm because communication latency can hinder the convergence of many algorithms in practice~\citep{TsianosNIPS2012}.  A natural way of limiting communication is to communicate only a fraction $\commfrac$ of the $T$ total iterations; at other times nodes simply perform local gradient steps.

We consider a sequence of i.i.d~random matrices $\{\bP(t)\}$ for Algorithm \ref{alg:DiSCO}
where $\bP(t) \in \{ \mbf{I},\bP \}$ with probabilities $1 - \commfrac$ and $\commfrac$, respectively, 
where $\mathbf{I}$ is the identity matrix (implying no communication since $P_{ij}(t)=0$ for $i \ne j$) and, as in the previous section, $\bP$ is a fixed doubly stochastic matrix respecting the graph constraints. For this model the expected number of times communication takes place is simply $\nu T$.   Note that now we have an additional randomization due to the Bernoulli distribution over the doubly stochastic matrices.  Analyzing a matrix $\bP(t)$ that depends on the current value of the iterates is considerably more complicated.

A straightforward application of Theorem \ref{theorem:mainThrmStoch} reveals that the optimization error is proportional to $\frac{1}{\nu}$ and decays as $\mc{O}(\frac{1}{\nu} \cdot \frac{\log^2(T)}{T})$. However, this ignores the effect of the local communication-free iterations. %

\textbf{A mini-batch approach.} To account for local communication free iterations we modify the intermittent communication scheme to follow a deterministic schedule of communication every $1/\nu$ steps. However, instead of taking single gradient steps between communication rounds, each node gathers the (sub)gradients and then takes an aggregate gradient step. That is, after the $t$-th round of communication, the node samples a batch $\mc{I}_t$ of indices sampled with replacement from its local data set with $|\mc{I}_t| = 1/\nu$. We can think of this as the base algorithm with a better gradient estimate at each step. The update rule is now
\begin{align}
\w_i(t+1) = \sum_{j \in \mc{N}_i} \w_j(t) P_{ij}(t) - \step_t \nu \sum_{ i \in \mc{I}_i } \g_i(t).
\label{eq:mbatchComm}
\end{align} 
We define $\g^{1/\nu}_i(t) = \sum_{ i \in \mc{I}_i } \g_i(t)$. Now the iteration count is over the communication steps and $\g^{1/\nu}_i(t)$ is the aggregated mini-batch (sub)gradient of size $1/\nu$. Note that this is analogous to the random scheme above but the analysis is more tractable.

\begin{theorem}
Fix a Markov matrix $\bP$ and let $\specnorm^2=\sigma_1(\mbf{\hat{\Sigma}})$ denote the spectral norm of the covariance matrix of the data distribution. Consider Algorithm \ref{alg:DiSCO} when the objective $J(\w)$ is strongly convex, $\bP(t) = \bP$ for all $t$, and $\step_t=1/(\reg t)$ for scheme \eqref{eq:mbatchComm}.  Let $\lambda_2(\bP)$ denote the second largest eigenvalue of $\bP$.  Then if the number of samples on each machine $n$ satisfies 
	\begin{align}
	n > \frac{4}{3 \rho^2} \log \left ( d\right ) 
	\end{align}
and 
	\begin{align}
	&T > \frac{2e}{\nu}  \log(1/\sqrt{ \lambda_2(\bP)}) \notag \\
	&\frac{T}{\log(\nu T)} > \max\left(\frac{4}{3\nu \specnorm^2}\log(d), \frac{\left( \frac{8}{5} \right)^{\frac{1}{4}} \sqrt{ m/\specnorm^2 }}{\log(1/\lambda_2)}\right) \notag \\
	&\frac{1}{\nu} > \frac{4}{3\rho^2} \cdot \log(d)
	 \label{eq:thm:Tbound}
	\end{align}
	and 
then the expected error for each node $i$ satisfies 
	\begin{align}
	&\E\left [ J(\hat{\w}_i(T)) -J(\w^{*})  \right ]  \notag \\
	&\le \left( \frac{1}{m} 
		+ 200\sqrt{5} 
			\cdot \frac{\sqrt{m\specnorm^4}\cdot \log (\nu T)}{1-\sqrt{\lambda_2}} 
		\right) \notag \\
		&\hspace{1in}\cdot \frac{L^2}{\reg} \cdot \frac{\log (\nu T)}{T}.
\end{align}
where $\nu$ is the frequency of communication and where $\lambda_2=\lambda_2(\bP)$.
\label{theorem:mainThrm:mbatch}
\end{theorem}

\noindent \textit{Remark:} Theorem \eqref{theorem:mainThrm:mbatch} suggests that if the inverse frequency of communication is large enough then we can obtain a sharper bound on the error by a factor of $\rho$. This is significantly better than a $\mc{O}(\sqrt{m\specnorm^2}\cdot\frac{\log \nu T}{\nu T})$ baseline guarantee from a direct application of Theorem \ref{theorem:mainThrm} when the number of iterations is $\nu T$.

Additionally the result suggests that if we communicate on a mini batch(where batch size $b=1/\nu$) that is large enough we can improve Theorem \ref{theorem:mainThrm}, specifically now we get a $1/m$ improvement when $m \leq 1/\rho^{4/3}$.

\section{Asymptotic Regime}
In this section we explore the sub-optimality of distributed primal averaging when $T \rightarrow \infty$ for the case of smooth strongly convex objectives. The results of Section \eqref{section:CommEveryIteration} suggest that we never gain from adding more machines in any network. Now we investigate the behaviour of Algorithm \ref{alg:DiSCO} in the asymptotic regime and show that the network effect disappears and we do indeed gain from more machines in any network.

Our analysis depends on the asymptotic normality of a variation of Algorithm \ref{alg:DiSCO}~\cite[Theorem 5]{BianchiFortHachem:13IEEETrans}. The main differences between Algorithm \ref{alg:DiSCO} and the consensus algorithm of Bianchi et al.~\cite{BianchiFortHachem:13IEEETrans} is that we average the iterates before making the local update. 

We make the following assumptions for the analysis in this section: (1) The loss function differentials $\{ \partial \left(\ell(\cdot) \right) \}$ are differentiable and $G$-Lipschitz for some $G>0$, (2) the stochastic gradients are of the form $\g_i(t)=\nabla J(\w_i(t)) + \mbs{\xi}_t$ where $\E[\mbs{\xi}_t] = \mbf{0}$ and $\E[\mbs{\xi}_t \mbs{\xi}^{\trans}_t] = \mbf{C}$, and (3) there exists $p>0$ such that $\E\left[ \norm{\mbs{\xi}_t}^{2+p} \right] < \infty$.
Our results hold for all smooth strongly convex objectives not necessarily dependent on $\w^{\trans}\x$.

\begin{lemma}   \label{lemma:mainThrmLarge}
Fix a Markov matrix $\bP$. Consider Algorithm \ref{alg:DiSCO} when the objective $J(\w)$ is strongly convex and twice differentiable, $\bP(t) = \bP$ for all $t$, and $\step_t=1/(\lambda t)$.  
then the expected error for each node $i$ satisfies for a arbitrary split of $N$ samples into $m$ nodes
\begin{align}
&\limsup\limits_{T\rightarrow \infty} T \cdot \E\left[J\left(\sum_{j=1}^m P_{ij} \w_j(T)\right) - J(\w^{*})\right] \notag \\
&\le
\sum_{j \in \mc{N}(i)}(P_{ij})^2 \cdot \Tr\left(\mbf{H}\right) \cdot \frac{G}{\reg}
\end{align}
where $\mbf{H}$ is the solution to the equation
\begin{align}
\nabla J^2(\w^{*}) \mbf{H} + \mbf{H} \nabla J^2(\w^{*})^{T} = \mbf{C}.
\end{align}
\end{lemma}

\noindent \textit{Remark:}  This result shows that asymptotically the network effect from Theorem \ref{theorem:mainThrmStoch} disappears and that additional nodes can speed convergence.

An application of Lemma \ref{lemma:mainThrmLarge} to problem \eqref{eq:optForm} gives us the following result for the specialized case of a $k$-regular graph with constant weight matrix $\bP$.

\begin{theorem}   \label{theorem:mainThrmLargeApplication}
Consider Algorithm \ref{alg:DiSCO} when the objective $J(\w)$ has the form \ref{eq:optForm}
, $\bP(t) = \bP$ and corresponds to a $k$-regular graph with uniform weights for all $t$, and $\step_t=1/(\lambda t)$.  
then the expected error for each node $i$ satisfies 
\begin{align}
&\limsup\limits_{T\rightarrow \infty} T \cdot \E\left[J\left(\sum_{j=1}^m P_{ij} \w_j(T)\right) - J(\w^{*})\right] \notag \\
&\le \frac{25\rho L^2}{k} \cdot \Tr\left(\nabla^2 J(\w^{*})^{-1} \right) \cdot \frac{G}{\reg} 
 \end{align}
where the expectation is with respect to the history of the sampled gradients as well as the uniform random splits of $N$ data points across $m$ machines.
\end{theorem}

\noindent \textit{Remark:}  For objective \eqref{eq:optForm} we obtain a $1/k$ variance reduction and the network effect disappears.

\section{Experiments}
\begin{figure*}[t]
\centering
\scalebox{0.9}{
\includegraphics[width=3.5in]{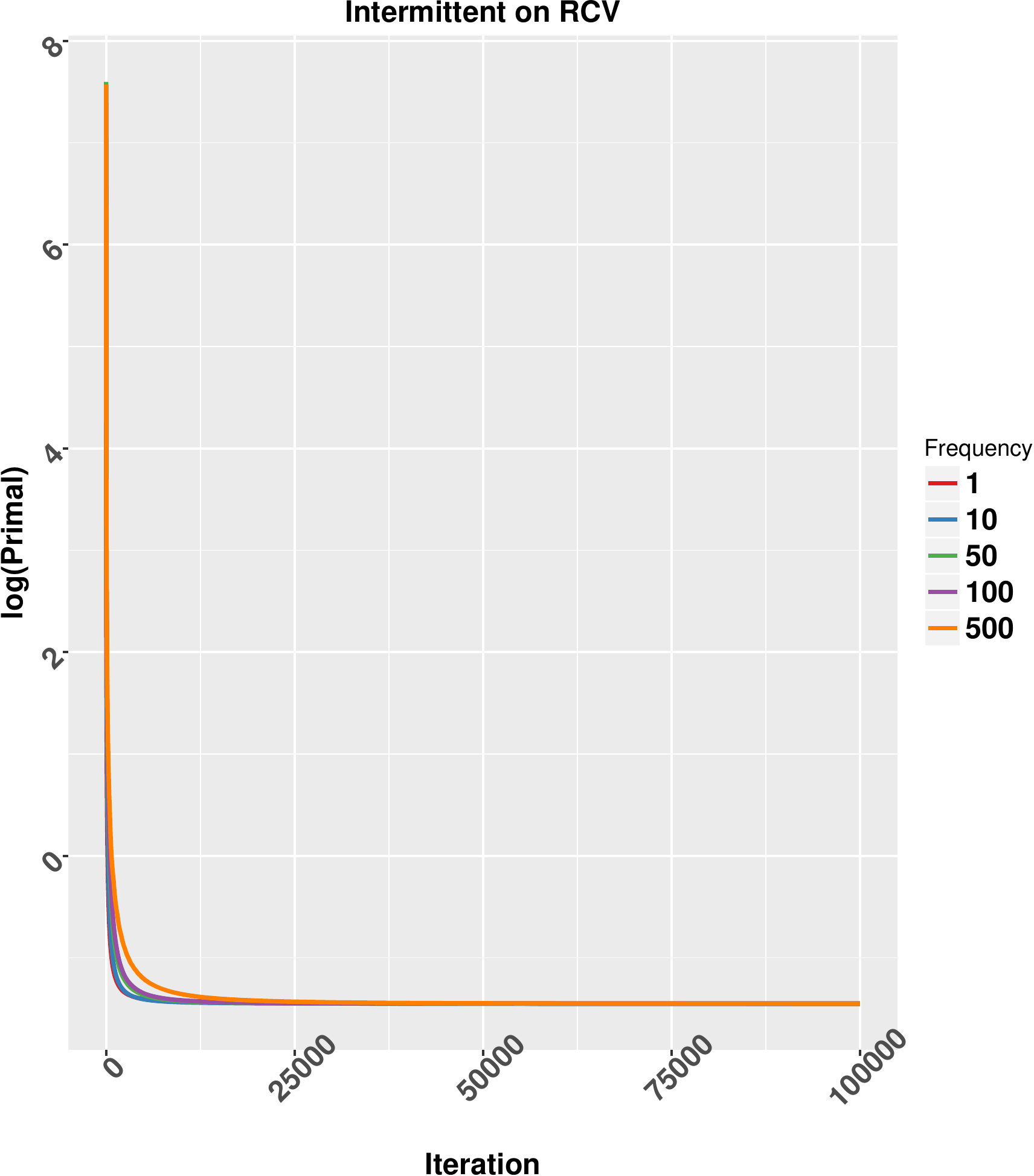}
\includegraphics[width=3.5in]{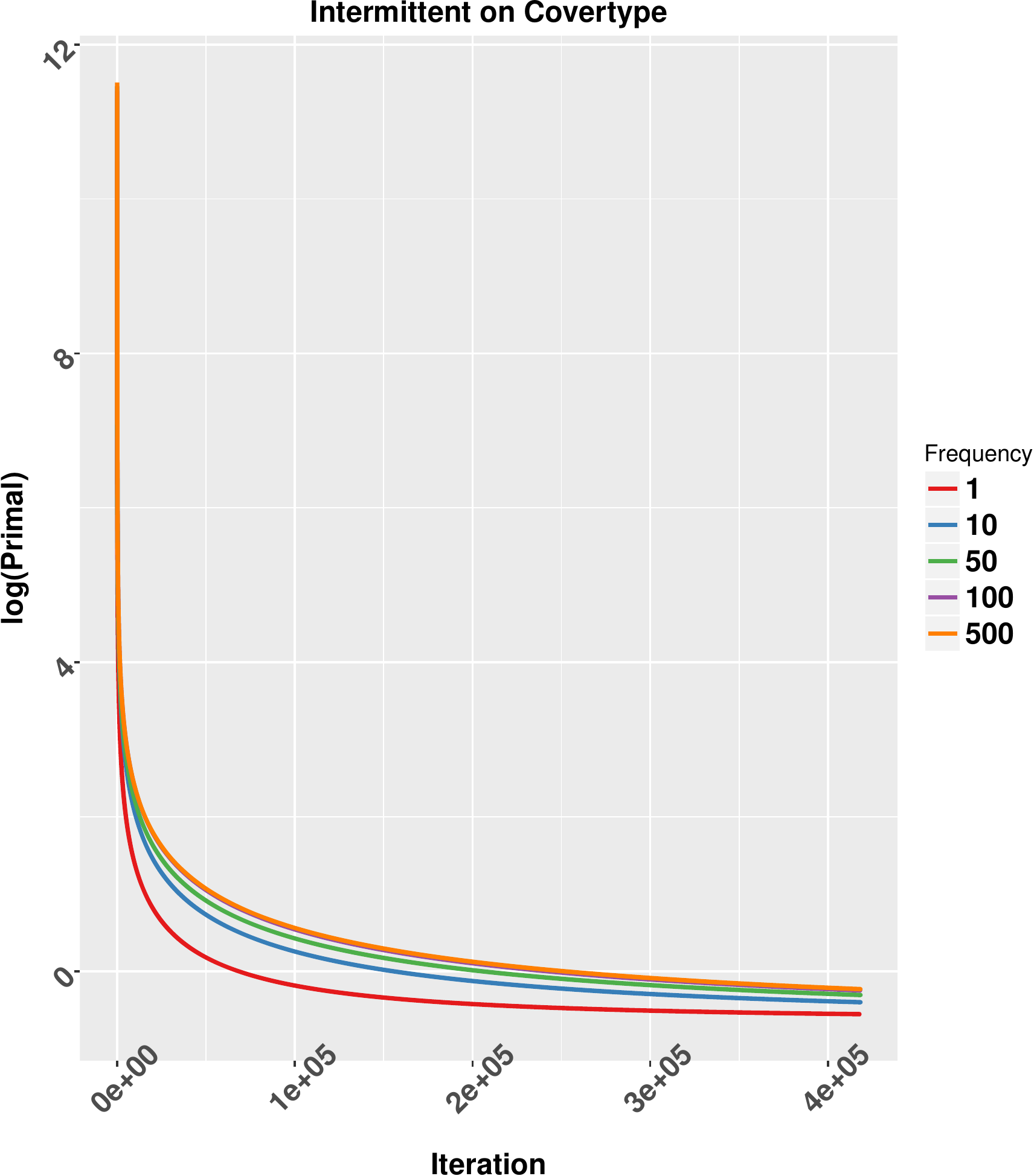}}
\label{fig:dist:cov}

\caption{ Performance of Algorithm \eqref{alg:DiSCO} with intermittent communication scheme on datasets with very different $\specnorm^2$. The algorithm works better for smaller $\specnorm^2$ and there is less decay in performance for \rcv~ as we decrease the number of communication rounds as opposed to \ctype~ ($\specnorm^2=0.01$ vs $\specnorm^2=0.21$).}
\label{fig:rcvcovInterm}
\end{figure*} 

Our goals in our experimental evaluation are to validate the theoretical dependence of the convergence rate on $\specnorm^2$ and to see if the conclusions hold when the assumptions we make in the analysis are violated. Note that all our experiments are based on simulations on a multicore computer.

\subsection{Data sets, tasks, and parameter settings} 

\begin{table}[t]
\centering
\caption{Data sets and parameters for experiments}
\begin{tabular}{l|c|c|c|c|c|}
data set & training  & test  & dim. & $\lambda$ & $\specnorm^2$ \\
\hline
\rcv & $781,265$ & $23,149$ & $47,236$ & $10^{-4}$ & $0.01$ \\
\ctype & $522,911$ & $58,001$ & $47,236$ & $10^{-6}$ & $0.21$ 
\end{tabular}
\label{tab:data}
\end{table}

The data sets used in our experiments are summarized in Table \eqref{tab:data}. \ctype \ is the forest covertype dataset~\cite{covertype} used in \cite{SSSC11:pegasos} obtained from the UC Irvine Machine Learning Repository~\cite{Lichman:2013}, and rcv1 is from the Reuters collection~\cite{Lichman:2013} obtained from libsvm collection \cite{ChangL:11libsvm}. The \rcv~ data set has a small value of $\hat{\rho}^2$, whereas \ctype \ has a larger value. \label{tab:data} In all the experiments we looked at $\ell_2$-regularized classification objectives for problem \eqref{eq:optForm}.   Each plot is averaged over $5$ runs.%

The data consists of pairs $\{(\x_1,y_1), \ldots, (\x_N,y_N)\}$ where $\x_i \in \mathbb{R}^d$ and $y_i \in \{-1,+1\}$. In all experiments we optimize the $\ell_2$-regularized empirical hinge loss where $\ell(\w^{\trans}\x ) = (1 - \w^{\trans} \x y)_{+}$.  %
The values of the regularization parameter $\reg$ are chosen from to be the same as those in Shalev-Shwarz et al.~\cite{SSSC11:pegasos}. 

We simulated networks of compute nodes of varying size ($m$) arranged in a $k$-regular graph with $k = \floor{0.25m}$ or a fixed degree (not dependent on $m$). Note that the dependence of the convergence rate of procedures like Algorithm \eqref{alg:DiSCO} on the properties of the underlying network has been investigated before and we refer the reader to Agarwal and Duchi~\cite{AgarwalD:11nips} for more details. In this paper we experiment only with $k$-regular graphs. The weights on the Markov matrix $\bP$ are set by using the max-degree Markov chain (see \citep{DiaconisBoyd}). One can also optimize for the fastest mixing Markov chain (\citep{DiaconisBoyd},\citep{DoubStoch::nips09}). Each node is randomly assigned $n=\floor{N/m}$ points.

\subsection{Intermittent Communication}

In this experiment we show the objective function for \rcv~ and \ctype~ as we change the frequency of communication (Figure \ref{fig:rcvcovInterm}), communicating after every $1,10,50$ and $500$ iterations. Indeed as predicted we see that the dataset with the larger $\rho^2$ appears to be affected more by intermittent communication. This indicates that network bandwidth can be conserved for datasets with a smaller $\rho^2$.

\begin{figure*}[t]
\centering
\scalebox{0.6}{
\includegraphics[width=\columnwidth]{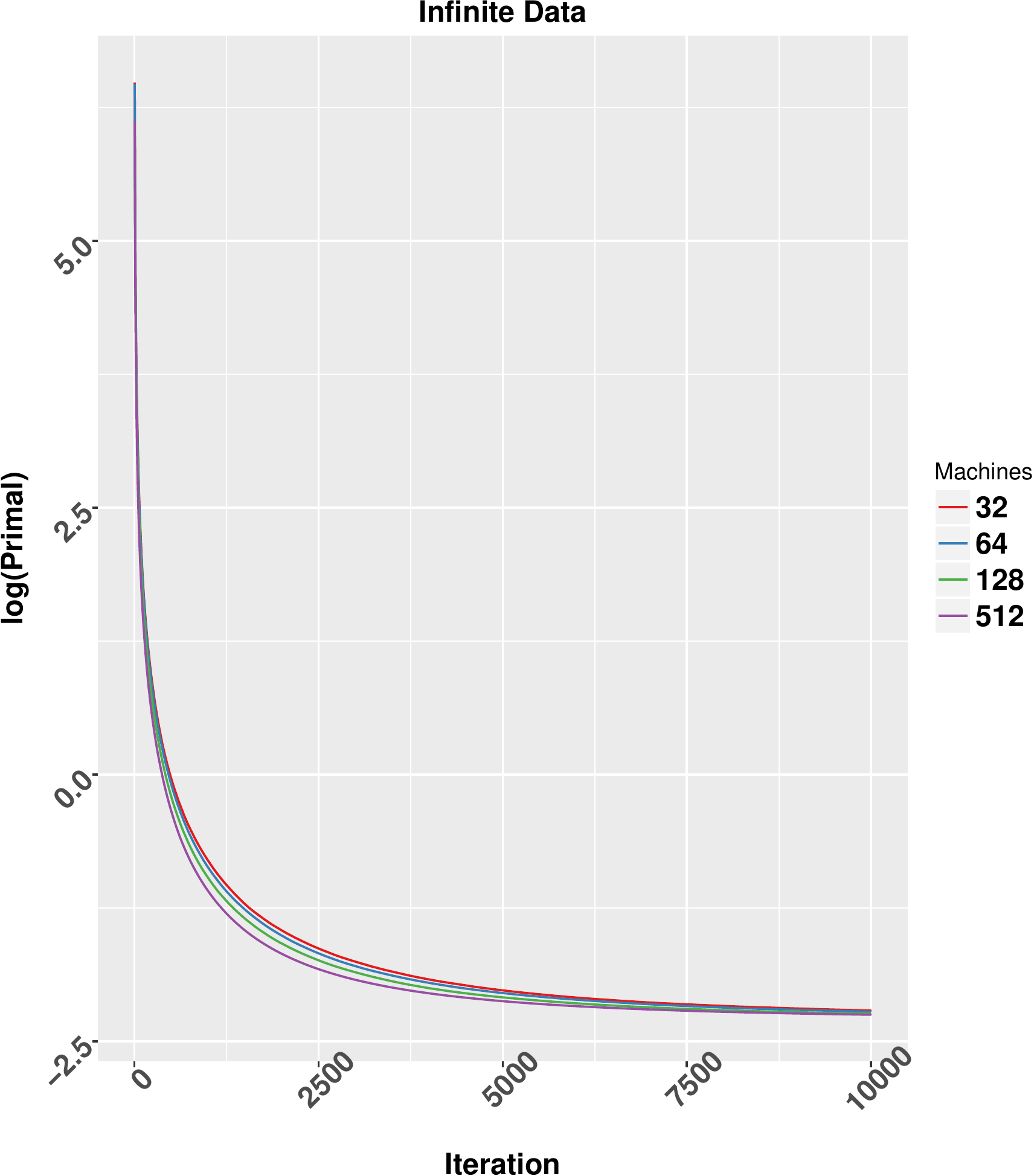}}
\label{fig:inifnite}
\caption{No network effect with increasing benefit of adding more machines in the case of infinite data.}
\label{fig:Infinite}
\end{figure*} 

\begin{figure*}
\centering
\scalebox{0.9}{
\includegraphics[width=3.8in]{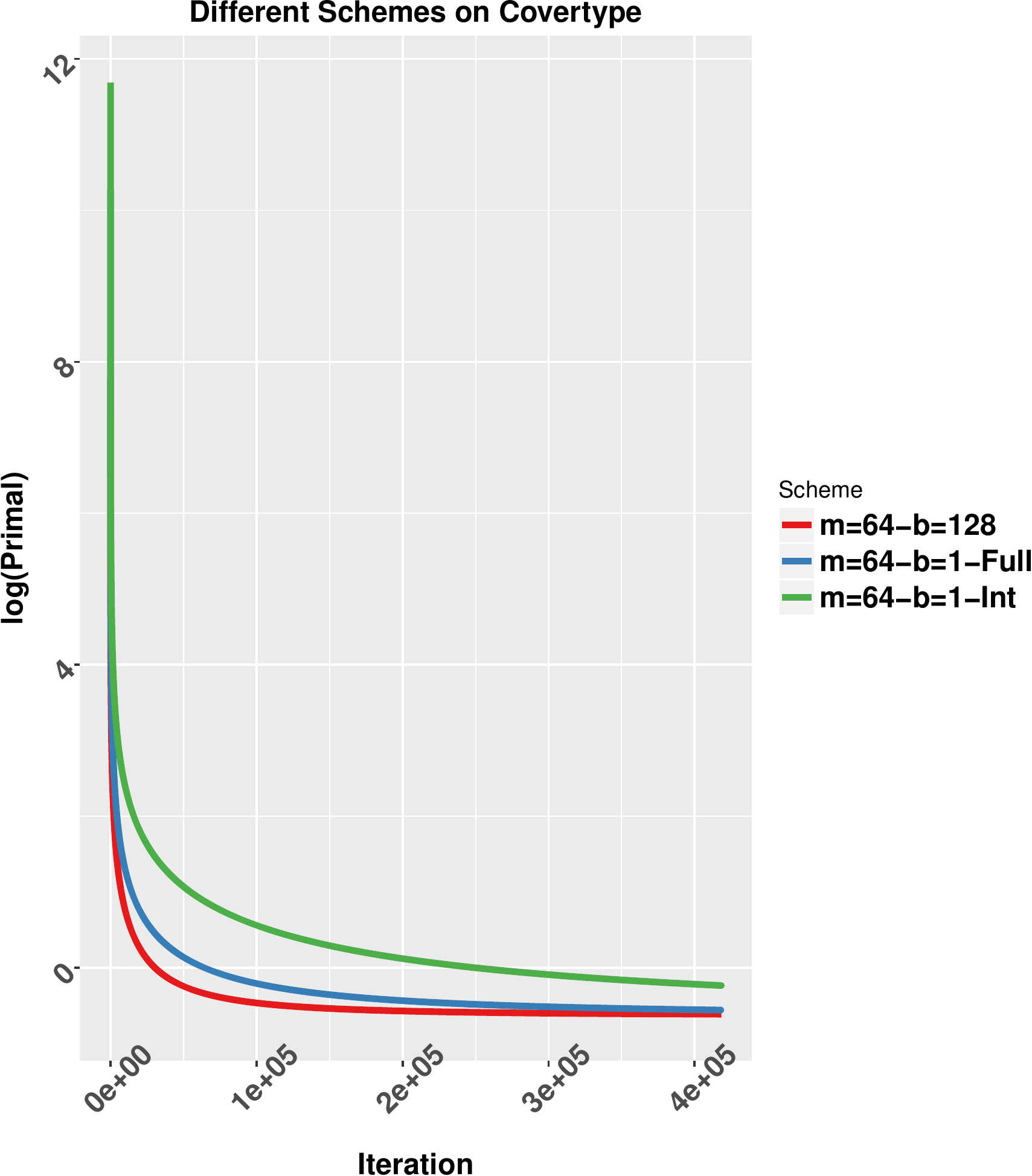}
\includegraphics[width=3.8in]{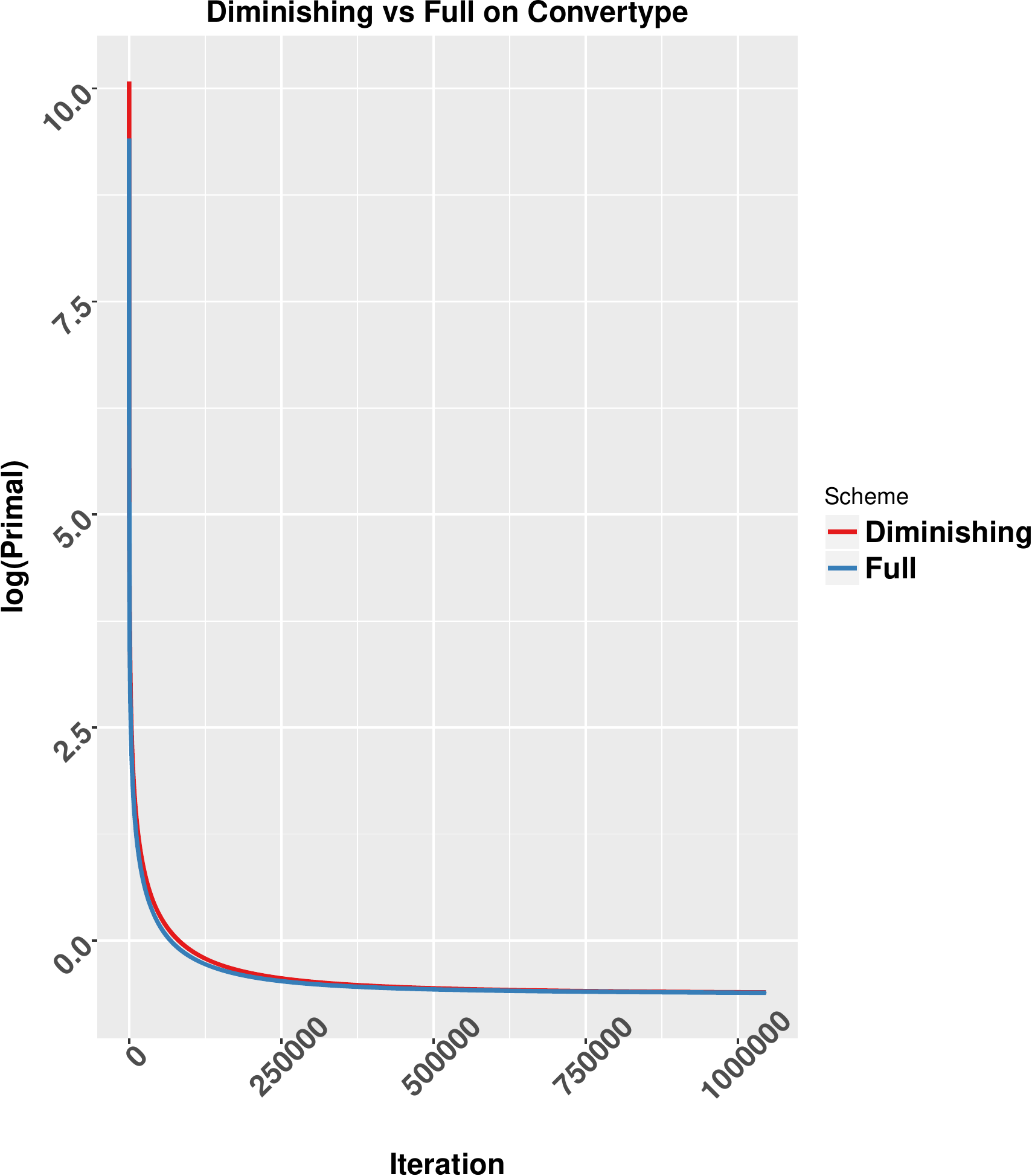}}
\caption{a) Comparison of three different schemes a) Algorithm \eqref{alg:DiSCO} with Mini-Batching b) Standard c) Intermittent with $b=(1/\nu)=128$. As predicted the mini-batch scheme performs much better than the others. b) The performance on Covertype with a full and a diminishing communication scheme is similar.}
\label{fig:rcvcovdiffschemes}
\end{figure*}
\subsection{Comparison of Different Schemes}
We compare the three different schemes proposed in this paper. On a network of $m=64$ machines we plot the performance of the mini batch extension of Algorithm \eqref{alg:DiSCO} with batch size $128$ against the intermittent scheme that communicates after every $128$ iterations and also the standard version of the algorithm. In Figure \ref{fig:rcvcovdiffschemes}-(a) we see that as predicted in Theorem \eqref{theorem:mainThrm:mbatch} the mini batch scheme proposed in \eqref{eq:mbatchComm} does better than the vanilla and the intermittent scheme.

\subsection{Infinite Data} 

To provide some empirical evidence of Lemma \ref{lemma:mainThrmLarge} we generate a very large ($N=10^7$) synthetic dataset from a multivariate Normal distribution and created a simple binary classification task using a random hyperplane. As we can see in figure \ref{fig:Infinite} for the SVM problem and a $k$-regular network we continue to gain as we add more machines and then eventually we stabilize but never lose from more machines. We only show the first few thousand iterations for clarity.
\removed{
\begin{figure*}[t]
\centering
\scalebox{0.7}{
\includegraphics[width=\columnwidth]{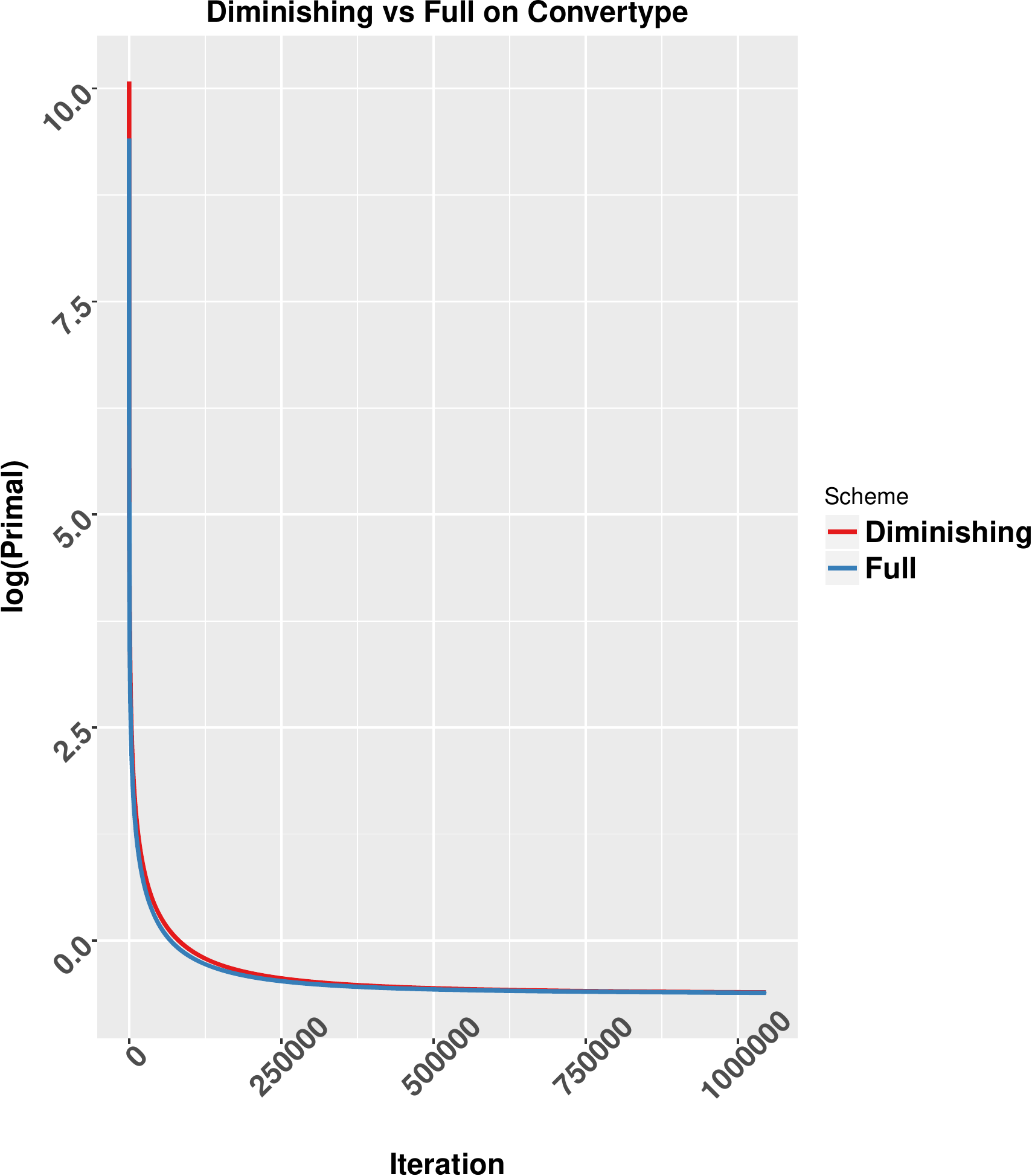}}
\label{fig:fullvsDim}
\caption{The performance on Covertype with a full and a diminishing communication scheme is similar.}
\label{fig:Dimin}
\end{figure*}} 
\subsection{Diminishing Communication}

To test if our conclusions apply when the i.i.d assumption for the matrices $\bP(t)$ does not hold we simulate a diminishing communication regime. Such a scheme can be useful when the nodes are already close to the optimal solution and communicating their respective iterate is wasteful. Intuitively it is in the beginning the nodes should communicate more frequently. To formalize the intuition we propose the following communication model
\begin{align}
\bP(t) =
\left\{
	\begin{array}{ll}
		\bP  & \mbox{w.p. } Ct^{-p} \\
		\mbf{I} & \mbox{w.p. } 1-Ct^{-p}
	\end{array}
\right.
\label{eq:powerLaw}
\end{align}
where $C,p > 0$. Thus the sequence of matrices are not identically distributed and the conclusions of Theorem \eqref{theorem:mainThrmStoch} do not apply.

However in Figure \ref{fig:rcvcovdiffschemes}-(b) (C=1,p=0.5) we see that on a network of $m=128$ nodes the performance for the diminishing regime is similar to the full communication case and we can hypothesize that our results also hold for non i.i.d communication matrices.
\clearpage

\section{Discussion and Implications}
In this paper we described a consensus stochastic gradient descent algorithm and analyzed its performance in terms of the spectral norm $\specnorm^2$ of the data covariance matrix under a homogenous assumption. In the consensus problem this setting has not been analyzed before and existing work corresponds to weaker results when this assumption holds.

For certain strongly convex objectives we showed that the objective value gap between any node's iterate and the optimum centralized estimate decreases as $O(\log^2(T)/T)$; crucially, the constant depended on $\specnorm^2$ and the spectral gap of the network matrix. \removed{This dependence on $\specnorm^2$ also appears in the analysis of mini-batching~\citep{TakacBRS:13icml}.} We showed how limiting communication can improve the total runtime and reduce network costs by extending our analysis with a similar data dependent bound. Moreover we show that in the asymptotic regime the network penalty disappears.
Our analysis suggests that distribution-dependent bounds can help us understand how data properties can mediate the tradeoff between computation and communication in distributed optimization.  In a sense, data distributions with smaller $\specnorm^2$ are easier to optimize over in a distributed setting. This set of distributions includes sparse data sets, an important class for applications. 

In the future we will extend data dependent guarantees to serial algorithms as well as the \textit{average-at-end} scheme~\citep{ZhangDW:12,ShamirSrebroZhang:14icml}. Extending our fixed batch-size to random size can help us understand the benefit of communication-free iterations. Finally, we can also study the impact of asynchrony and more general time-varying topologies.

\section{Appendix}

We gather here the proof details and technical lemmas needed to establish our results.

\section{Proof of Theorem \ref{theorem:mainThrm}}

Theorem \ref{theorem:mainThrm} provides a bound on the suboptimality gap for the output $\hat{\w}_i(T)$ of Algorithm \ref{alg:DiSCO} at node $i$, which is the average of that node's iterates. In the analysis we relate this local average to the average iterate across nodes at time $t$:
	\begin{align}
	\bar{\w}(t)  = \sum_{i=1}^m \frac{\w_i(t)}{m}.
	\label{eq:node_avg_iterate}
	\end{align}
We will also consider the average of $\bar{\w}(t)$ over time. 

The proof consists of three main steps. 
\begin{itemize}
\item We establish the following inequality for the objective error:
	\begin{align}
	\E \left[ J(\bar{\w}(t))- J(\w^{*}) \right] &\le \notag \\
   	&\hspace{-1in}
	\frac{(\step_t^{-1}-\reg)}{2} \E \left[ \norm{\bar{\w}(t)-\w^{*}}^2 \right] \notag \\
   	&\hspace{-1in}
	- \frac{\step_t^{-1}}{2} \E\left[ \norm{\bar{\w}(t+1)-\w^{*}}^2 \right] 		
		\notag \\
	&\hspace{-1in} 
	+ \frac{\step_t}{2} \E\left[ \norm{\sum_{i=1}^{m}\frac{\g_i(t)}{m}}^2 
			\right] \notag \\ 
	&\hspace{-1in} 
	+\sum_{i=1}^{m} 
			\sqrt{\E\left[\norm{ \bar{\w}(t)-\w_i(t) }^2 \right]} \notag \\
	&\hspace{-1in}
	\cdot \sqrt{\E\left[ \left(\norm{\nabla J_i(\w_i(t))} 
				+ \norm{\nabla J_i(\bar{\w}(t))} \right)^2 \right] }/m,
	 \label{eq:mainBnd_before}
	 \end{align}
where $\bar{\w}(t)$ is the average of the iterates at all nodes and the expectation is with respect to $\calF_t$ while conditioned on the sample split across nodes. All expectations, except when explicitly stated, will be conditioned on this split.
 \item We bound $\E\left[\norm{\nabla J(\w_i(t))}^2\right] $ and $\frac{\step_t}{2}\E \left[ \norm{\sum_{i=1}^{m}\frac{\g_i(t)}{m}}^2 \right]$ in terms of the spectral norm of the covariance matrix of the distribution $\mc{P}$ by additionally taking expectation with respect to the sample $S$.
 \item We bound the network error $\E\left[\norm{ \bar{\w}(t)-\w_i(t) }^2 \right]$ in term of the network size $m$ and a spectral property of the matrix $\bP$.
\end{itemize}
Combining the bounds using inequality \eqref{eq:mainBnd_before} and applying the definition of subgradients yields the result of Theorem \ref{theorem:mainThrm}.

\subsection{Spectral Norm of Random Submatrices}
In this section we establish a Lemma pertaining to the spectral norm of submatrices that is central to our results. Specifically we prove the following inequality, which follows by applying the Matrix Bernstein inequality of Tropp~\cite{Tropp:12tail}.

\begin{lemma}\label{lem:specnormIntdim}
Let $\mc{P}$ be a distribution on $\mathbb{R}^d$ with second moment matrix $\mbs{\Sigma} = \E_{\Z \sim \mc{P}}[ \Z \Z^{\trans}]$ such that $\norm{\Z_k} \le 1$ almost surely.  Let $\dists^2 = \spec(\mbs{\Sigma})$.  Let $\Z_1, \Z_2, \ldots, \Z_K$ be an i.i.d. sample from $\mc{P}$ and let
	\[
	\Q_K = \sum_{k=1}^{K} \Z_k \Z_k^{\trans}
	\]
be the empirical second moment matrix of the data.  Then for $K > \frac{4}{3\dists^2} \log(d)$,
	\begin{align}
	\E\left[ \frac{ \spec( \Q_K ) }{K} \right] 
	&\le 5 \dists^2. 
	\end{align}
\end{lemma}

Thus when $\mc{P}$ is the empirical distribution we get that $	\E\left[ \frac{ \spec( \Q_K ) }{K} \right] \le 5 \dists^2$.

\noindent \textit{Remark:} We can replace the ambient dimension $d$ in the requirement for $K$ by an intrinsic dimensionality term but this requires a lower bound on the norm of any data point in the sample. 

\begin{proof}
Let $\Z$ be the $d \times K$ matrix whose columns are $\{\Z_k\}$.  Define $\X_k = \Z_k \Z_k^{\trans} - \mbs{\Sigma}$.  Then $\E[ \X_k ] = \mbf{0}$ and
	\begin{align*}
	\lmax(\X_k) &= \lmax\left( \Z_k \Z_k^{\trans} - \mbs{\Sigma} \right) \\
	&\le \norm{ \Z_k }^2 \\
	&\le 1,
	\end{align*}
because $\mbs{\Sigma}$ is positive semidefinite and $\norm{ x_i } \le 1$ for all $i$.  Furthermore,
	\begin{align*}
	\spec\left( \sum_{k=1}^{K} \E\left[ \X_k^2 \right] \right) 
	&= K \spec\left(  
		\E\left[ \Z_k \Z_k^{\trans} \Z_k \Z_k^{\trans} \right] - \mbs{\Sigma}^2 \right) \\
	&\le K \spec\left( \E\left[ \norm{Y_k}^2 \Z_k \Z_k^{\trans} \right] \right) 
		+ K \spec\left( \mbs{\Sigma} \right)^2 \\
	&\le K (\dists^2 + \dists^4) \\
	&\le 2 K \dists^2,
	\end{align*}
since $\specnorm \le 1$.
	
Applying the Matrix Bernstein inequality of Tropp~\cite[Theorem 6.1]{Tropp:12tail}:
	\begin{align}
	\Prob\left( \spec\left( \sum_{k=1}^{K} \mbf{X}_k \right) \ge r \right) 
	\le %
		\begin{cases}
		d \exp\left( - 3 \frac{r^2}{ 16 K \dists^2} \right) & \frac{r}{K} \le 2 \dists^2 \\
		d \exp\left( - 3 \frac{r}{8} \right) & \frac{r}{K} \ge 2 \dists^2
		\end{cases}.
	\label{eq:bernstein1}
	\end{align}
Now, note that
	\begin{align*}
	\spec\left( \sum_{k=1}^{K} \X_k  \right) = \spec\left( \sum_{k=1}^{K} \Z_k \Z_k^{\trans} - K \mbs{\Sigma}  \right),
	\end{align*}
so $\spec\left( \sum_{k=1}^{K} \mbf{X}_k \right) \ge r$ is implied by 
	\[
	\left| \frac{1}{K} \spec\left( \sum_{k=1}^{K} \Z_k \Z_k^{\trans} \right) - \spec\left( \mbs{\Sigma}  \right) \right| \ge \frac{r}{K}.  
	\]
Therefore
	\begin{align}
	\Prob\left( \left| \frac{ \spec(\Q_K) }{ K } - \dists^2 \right| \ge r' \right)
	\le
		\begin{cases}
		d \exp\left( - \frac{3 K r'^2}{ 16 \dists^2} \right) & r' \le 2 \dists^2 \\
		d \exp\left( - \frac{3 K r'}{8} \right) & r' \ge 2 \dists^2
		\end{cases}.
	\label{eq:bernstein2}
	\end{align}
Integrating \eqref{eq:bernstein2} yields
	\begin{align*}
	&\E\left[ \frac{ \spec( \Q_K ) }{K} \right] = \notag \\
	&\int_{0}^{\infty} \Prob\left( \frac{ \spec(\Q_K) }{ K } \ge x \right) dx \\
	&\le 3 \dists^2 + \int_{3 \dists^2}^{\infty} \Prob\left( \frac{ \spec(\Q_t) }{ K } - \dists^2 \ge x - \dists^2 \right) dx \\
	&\le 3 \dists^2 + \int_{2 \dists^2}^{\infty} \Prob\left( \frac{ \spec(\Q_t) }{ K } - \dists^2 \ge r' \right) dr' \\
	&\le 3 \dists^2 + \int_{2 \dists^2}^{\infty} d \exp\left( - \frac{3}{8} K r' \right) dr' \\
	&= 3 \dists^2 + \frac{8}{3} \cdot \frac{d}{K} \exp\left( - \frac{3}{4} \dists^2 K \right)
	\end{align*}
For $K > \frac{4}{3 \dists^2} \log d$,
	\begin{align*}
	\E\left[ \frac{ \spec( \Q_K ) }{K} \right] 
	&\le 3 \dists^2 + \frac{8}{3} \cdot \frac{3}{4} \cdot \frac{\dists^2}{\log d} \\
	&\le 5 \dists^2.
	\end{align*}	
\end{proof}
\subsection{Decomposing the expected suboptimality gap}

The proof in part follows \citep{nedicDistributedOptimization}. It is easy to verify that because $\bP$ is doubly stochastic the average of the iterates across the nodes at time $t$, the average of the iterates across the nodes in \eqref{eq:node_avg_iterate} satisfies the following update rule:
\vspace{-2mm}
\begin{align}
\bar{\w}(t+1) = \bar{\w}(t) - \step_t\sum_{i=1}^{m}\frac{\g_i(t)}{m}.
\label{eq:avgUp}
\vspace{-2mm}
\end{align}
We emphasize that in Algorithm \ref{alg:DiSCO} we do not perform a final averaging across nodes at the end as in \eqref{eq:node_avg_iterate}. Rather, we analyze the average at a single node across its iterates (sometimes called Polyak averaging).  Analyzing \eqref{eq:node_avg_iterate} provides us with a way to understand how the objective $J(\w_i(t))$ evaluated at any node $i$'s iterate $\w_i(t)$ compares to the minimum value $J(\w^*)$.  The details can be found in Section \eqref{sec:combo_bound}.

To simplify notation, we treat all expectations as conditioned on the sample $S$. Then \eqref{eq:avgUp},
\begin{align}
\E\left[ \norm{\bar{\w}(t+1)-\w^{*}}^2 \Big| \calF_t \right] & \notag \\
&\hspace{-1.2in}
=\E \left[ \norm{\bar{\w}(t)-\w^{*}}^2 | \calF_t \right] \notag \\
	&\hspace{-0.8in} 
	+ \step_t^2\E \left[ \norm{\sum_{i=1}^{m}\frac{\g_i(t)}{m}}^2 \Big| \calF_t\right] 
	\notag \\
	&\hspace{-0.8in}
	- 2\step_t(\bar{\w}(t)-\w^{*})^{\trans}\sum_{i=1}^{m}\frac{\E \left[ \g_i(t)|\calF_t \right]}{m} 
	\notag \\
&\hspace{-1.2in}
=  \E \left[ \norm{\bar{\w}(t)-\w^{*}}^2 |\calF_t \right] \notag \\
	&\hspace{-0.8in} 
	+ \step_t^2\E \left[ \norm{\sum_{i=1}^{m}\frac{\g_i(t)}{m}}^2 \Big| \calF_t \right] \notag \\
	&\hspace{-0.8in} 
	- 2\step_t\sum_{i=1}^{m}(\bar{\w}(t)-\w^{*})^{\trans}\frac{\E \left[ \g_i(t)|\calF_t \right]}{m}. 
\label{eq:comm_recursion}
\end{align}

Note that $\nabla J_i(\w_i(t)) = \E \left[ \g_i(t)|\calF_t \right]$, so for the last term, for each $i$ we have
\begin{align}
\nabla J_i(\w_i(t))^{\trans}(\bar{\w}(t)-\w^{*}) \notag \\
&\hspace{-1.2in}
= 
\nabla J_i(\w_i(t))^{\trans}\left(\bar{\w}(t)-\w_i(t)\right) 
	\notag \\
	&\hspace{-0.8in}
	+ \nabla J_i(\w_i(t))^{\trans}\left(\w_i(t) - \w^{*}\right) 
	\notag \\
&\hspace{-1.2in} 
\ge 
-\norm{\nabla J_i(\w_i(t))} \norm{\bar{\w}(t)-\w_i(t)} 	
	\notag \\
	&\hspace{-0.8in}
	+ \nabla J_i(\w_i(t))^{\trans}\left(\w_i(t) - \w^{*}\right) 
	\notag \\
&\hspace{-1.2in}
\ge
-\norm{\nabla J_i(\w_i(t))} \norm{\bar{\w}(t)-\w_i(t)} 
	\notag \\
	&\hspace{-0.8in}
	+ J_i(\w_i(t)) - J_i(\w^*)  
	+ \frac{\reg}{2}\norm{\w_i(t)-\w^{*}}^2
	\notag \\
&\hspace{-1.2in}
= 
-\norm{\nabla J_i(\w_i(t))} \norm{\bar{\w}(t)-\w_i(t)} 
	\notag \\
	&\hspace{-0.8in}
	+ J_i(\w_i(t)) - J_i(\bar{\w}(t)) 
	\notag \\
  	&\hspace{-0.8in} 
	+ \frac{\reg}{2}\norm{\w_i(t)-\w^{*}}^2 
	+  J_i(\bar{\w}(t)) - J_i(\w^*) 
	\notag \\
&\hspace{-1.2in}
\ge 
	-\norm{\nabla J_i(\w_i(t))} \norm{\bar{\w}(t)-\w_i(t)} 
	\notag \\
	&\hspace{-0.8in}
	+\nabla J_i(\bar{\w}(t))^{\trans}\left(\w_i(t) - \bar{\w}(t)\right) 
	\notag \\
  	&\hspace{-0.8in}
	+ \frac{\reg}{2}\norm{\w_i(t)-\w^{*}}^2 
	+  J_i(\bar{\w}(t)) - J_i(\w^*) 
	\notag \\
&\hspace{-1.2in}
\ge 
	-\left(\norm{\nabla J_i(\w_i(t))}
		+\norm{\nabla J_i(\bar{\w}(t))}\right) \norm{\bar{\w}(t)-\w_i(t)} 
	\notag \\
 	&\hspace{-0.8in}
	+ \frac{\reg}{2}\norm{\w_i(t)-\w^{*}}^2 
	+ J_i(\bar{\w}(t)) - J_i(\w^*),
\end{align}
where the second and third lines comes from applying the Cauchy-Shwartz inequality and strong convexity, the fifth line comes from the definition of subgradient, and the last line is another application of the Cauchy-Shwartz inequality.

Averaging over all the nodes, using convexity of $\norm{\cdot}^2$, the definition of $J(\cdot)$, and Jensen's inequality yields the following inequality:
\begin{align}
-2 \step_t \sum_{i=1}^{m} (\bar{\w}(t)-\w^{*})^{\trans} \frac{\E[\g_i(t)|\calF_t]}{m} 	
	\notag \\
&\hspace{-2in}
\le 
	2\step_t  
	\sum_{i=1}^{m}
		\frac{
			\norm{ \bar{\w}(t)-\w_i(t) } 
			\left (\norm{\nabla J_i(\w_i(t))} 
				+ \norm{\nabla J_i(\bar{\w}(t))} \right )
			}{m}
	\notag \\
	&\hspace{-1.7in}
	-2\step_t
	\left( \sum_{i=1}^m \frac{J_i(\bar{\w}(t))- J_i(\w^{*})}{m} \right) 
	\notag \\
	&\hspace{-1.7in}
	-\reg\step_t \sum_{i=1}^{m}\frac{\norm{\w_i(t)-\w^{*}}^2}{m} 
	\notag \\
&\hspace{-2in}
\le
	2\step_t 
	\sum_{i=1}^{m}\frac{\norm{ \bar{\w}(t)-\w_i(t) } \left (\norm{\nabla J_i(\w_i(t))} + \norm{\nabla J_i(\bar{\w}(t))} \right )}{m}
	\notag \\
	&\hspace{-1.7in}
	-2\step_t\left( J(\bar{\w}(t))- J(\w^{*}) \right) - \reg\step_t \norm{\bar{\w}(t)-\w^{*}}^2
\label{eq:node_average}
\end{align}

Substituting inequality \eqref{eq:node_average} in recursion \eqref{eq:comm_recursion}, 
	\begin{align}
	\E\left[ \norm{\bar{\w}(t+1)-\w^{*}}^2 \big| \calF_t \right] 
	\notag \\
	&\hspace{-1.6in}
	\le
		\E \left[ \norm{\bar{\w}(t)-\w^{*}}^2 |\calF_t \right] 
		\notag \\
  		&\hspace{-1.5in}
		+ \step_t^2 \E\left[ \norm{\sum_{i=1}^{m}\frac{\g_i(t)}{m}}^2 
			~\Big|~ \calF_t\right] 
		\notag \\
		&\hspace{-1.5in} 
		+ 2 \step_t
		\sum_{i=1}^{m}
			\frac{
				\norm{ \bar{\w}(t)-\w_i(t) } 
				\left(\norm{\nabla J_i(\w_i(t))} 
					+ \norm{\nabla J_i(\bar{\w}(t))}
				\right)
			}{m} 
		\notag \\
		&\hspace{-1.5in} 
 		-2 \step_t \left( J(\bar{\w}(t))- J(\w^{*}) \right) 
		- \reg \step_t \norm{\bar{\w}(t)-\w^{*}}^2.
	\label{eq:recurse1}
	\end{align}
Taking expectations with respect to the entire history $\mc{F}_t$,
\begin{align}
\E\left[ \norm{\bar{\w}(t+1)-\w^{*}}^2 \right]  
	\notag \\
&\hspace{-1.3in}
\le
	\E\left[ \norm{\bar{\w}(t)-\w^{*}}^2 \right] 
	 + \step_t^2\E \left[ \norm{\sum_{i=1}^{m}\frac{\g_i(t)}{m}}^2 \right] 	
	\notag \\ 
	&\hspace{-1.2in}
	+ 2\step_t \cdot 
	\notag \\
	&\hspace{-1.1in}
		\sum_{i=1}^{m} \frac{
			\E\left[\norm{ \bar{\w}(t)-\w_i(t) } 
				\left(\norm{\nabla J_i(\w_i(t))} 
				+ \norm{\nabla J_i(\bar{\w}(t))} \right)
				\right]
			}{m} 
	\notag \\ 
	&\hspace{-1.2in}
	-2\step_t \left( \E\left[ J(\bar{\w}(t))- J(\w^{*}) \right] \right)
			- \reg\step_t \E \left[ \norm{\bar{\w}(t)-\w^{*}}^2 \right] 
	\notag \\ 
&\hspace{-1.3in}
\le 
	-2 \step_t \left( \E \left[ J(\bar{\w}(t))- J(\w^{*}) \right] \right)
	\notag \\ 
	&\hspace{-1.2in}
	+ (1 - \reg \step_t) \E\left[ \norm{\bar{\w}(t)-\w^{*}}^2 \right] 
	+ \step_t^2\E \left[ \norm{\sum_{i=1}^{m}\frac{\g_i(t)}{m}}^2 \right] 
	\notag \\ 
	&\hspace{-1.2in}
	+ \frac{2\step_t}{m} 
			 \sum_{i=1}^{m} 
			\sqrt{\E\left[\norm{ \bar{\w}(t)-\w_i(t) }^2 \right]} 
			\notag \\ 
			&\hspace{-0.6in}
			\cdot \sqrt{\E\left[ \left(\norm{\nabla J_i(\w_i(t))} 
				+ \norm{\nabla J_i(\bar{\w}(t))} \right)^2 \right] }
\end{align}

This lets us bound the expected suboptimality gap $\E \left[ J(\bar{\w}(t))- J(\w^{*}) \right]$ via three terms:
	\begin{align}
	\text{T1}
	&= \frac{(\step_t^{-1}-\reg)}{2}
			\E \left[ \norm{\bar{\w}(t)-\w^{*}}^2 \right]
		\notag \\
		&\hspace{0.5in} 
		- \frac{\step_t^{-1}}{2}
			\E\left[ \norm{\bar{\w}(t+1)-\w^{*}}^2 \right]
	\label{eq:T1}
	\\
	\text{T2}
	&=
		\frac{\step_t}{2}\E \left[ \norm{\sum_{i=1}^{m}\frac{\g_i(t)}{m}}^2 \right]
	\label{eq:T2}
	\\
	\text{T3}
	&= \frac{1}{m} \sum_{i=1}^{m} 
			\sqrt{\E\left[\norm{ \bar{\w}(t)-\w_i(t) }^2 \right]} 
		\notag \\
		&\hspace{0.5in} 
		\cdot 
		\sqrt{\E\left[ \left (\norm{\nabla J_i(\w_i(t))} 
				+ \norm{\nabla J_i(\bar{\w}(t))} \right)^2 \right] },
	\label{eq:T3}
	\end{align}
where
\vspace{-5mm}
\begin{align}
\E \left[ J(\bar{\w}(t))- J(\w^{*}) \right]  
&\le \text{T1} + \text{T2} + \text{T3}.
 \label{eq:mainBnd}
 \end{align}
 The remainder of the proof is to bound these three terms separately.
 \vspace{-2mm}
 
\subsection{Network Error Bound} 
We need to prove an intermediate bound first to handle term \text{T3}.

\begin{lemma}\label{lemma:avgdevBnd}
Fix a Markov matrix $\bP$ and consider Algorithm \ref{alg:DiSCO} when the objective $J(\w)$ is strongly convex we have the following inequality for the expected squared error between the iterate $\w_i(t)$ at node $i$ at time $t$ and the average $\bar{\w}(t)$ defined in Algorithm \ref{alg:DiSCO}:
	\begin{align}
	\sqrt{\E\left[\norm{ \bar{\w}(t)-\w_i(t) }^2 \right]} \le \frac{2L}{\reg}\cdot \frac{\sqrt{m}}{b }\cdot\frac{\log(2bet^2)}{t},
	\end{align}
where $b= (1/2)\log(1/\lambda_2(\bP))$.
\end{lemma}

\begin{proof}
We follow a similar analysis as others \citep[Prop. 3]{nedicDistributedOptimization}~\citep[IV.A]{dualAveraging}~\citep{DistStronglyConvex}.  Let $\mbf{W}(t)$ be the $m \times d$ matrix whose $i$-th row is $\w_i(t)$ and $\mbf{G}(t)$ be the $m \times d$ matrix whose $i$-th row is $\g_i(t)$ .  Then the iteration can be compactly written as
	\[
	\mbf{W}(t+1) = \bP(t) \mbf{W}(t) - \step_t \mbf{G}(t)
	\]
and the network average matrix $\bar{\mbf{W}}(t) = \frac{1}{m} \mbf{1} \mbf{1}^{\top} \mbf{W}(t)$.  Then we can write the difference using the fact that $\bP(t) = \bP$ for all $t$:
	\begin{align}
	&\bar{\mbf{W}}(t+1) - \mbf{W}(t+1) =\notag \\
		&\left(  \frac{1}{m} \mbf{1} \mbf{1}^{\top} - I \right) \left( \bP \mbf{W}(t) - \step_t \mbf{G}(t) \right) \notag \\
			&\hspace{-1in} \notag \\
		&= \left(  \frac{1}{m} \mbf{1} \mbf{1}^{\top} - \bP \right) \mbf{W}(t)
			- \step_t \left(  \frac{1}{m} \mbf{1} \mbf{1}^{\top} - I \right) \mbf{G}(t) \notag \\
		&\hspace{-1in} \notag \\
		&= \left(  \frac{1}{m} \mbf{1} \mbf{1}^{\top} - \bP \right) \left( \bP \mbf{W}(t-1) - \step_{t-1} \mbf{G}(t-1) \right)
			\notag \\
			&\quad - \step_t \left(  \frac{1}{m} \mbf{1} \mbf{1}^{\top} - I \right) \mbf{G}(t) \notag \\
				&\hspace{-1in} \notag \\
		&= \left(  \frac{1}{m} \mbf{1} \mbf{1}^{\top} - \bP^2 \right) \mbf{W}(t-1) \notag \\
			&\quad- \step_{t-1} \left(  \frac{1}{m} \mbf{1} \mbf{1}^{\top} - \bP \right) \mbf{G}(t-1)
			\notag \\
			&\quad- \step_t \left(  \frac{1}{m} \mbf{1} \mbf{1}^{\top} - I \right) \mbf{G}(t) \notag \\
				&\hspace{-1in} \notag \\
		&= \left(  \frac{1}{m} \mbf{1} \mbf{1}^{\top} - \bP^2 \right) \mbf{W}(t-1)\notag \\
		&\quad- \sum_{s=t-1}^{t} \step_s \left(  \frac{1}{m} \mbf{1} \mbf{1}^{\top} - \bP^{t-s} \right) \mbf{G}(s).
	\end{align}
Continuing the expansion and using the fact that $\mbf{W}(1) = \mbf{0}$,
	\begin{align}
   &	\bar{\mbf{W}}(t+1) - \mbf{W}(t+1) = \notag \\
	&\left( \frac{1}{m} \mbf{1} \mbf{1}^{\top} - \bP^t \right)  \mbf{W}(1)
		- \sum_{s=1}^{t} \step_s \left(  \frac{1}{m} \mbf{1} \mbf{1}^{\top} - \bP^{t-s} \right) \mbf{G}(s) \notag \\
		&\quad= - \sum_{s=1}^{t} \step_s \left(  \frac{1}{m} \mbf{1} \mbf{1}^{\top} - \bP^{t-s} \right) \mbf{G}(s) \notag \\
		&\quad= - \sum_{s=1}^{t-1} \step_s \left(  \frac{1}{m} \mbf{1} \mbf{1}^{\top} - \bP^{t-s} \right) \mbf{G}(s) \notag \\
		&\quad- \step_t \left(  \frac{1}{m} \mbf{1} \mbf{1}^{\top} - I \right) \mbf{G}(t).
			\label{eq:more:neterr_matrix}
	\end{align}
Now looking at the norm of the $i$-th row of \eqref{eq:more:neterr_matrix} and using the bound on the gradient norm:
	\begin{align}
	&\norm{ \bar{\w}(t)-\w_i(t) } \notag \\
		&\quad \le \Bigg\| 
			\sum_{s=1}^{t-1} \step_s \sum_{j=1}^{m} 
				\left( \frac{1}{m} - (\bP^{t-s})_{ij} \right) \g_j(s) 
			\notag \\
			&\hspace{1in}
			+ \step_t \left( \sum_{j=1}^{m} \frac{1}{m} \g_j(t) - \g_i(t) \right) 
				\Bigg\| \\
		&\le \sum_{s=1}^{t-1} \frac{L}{\reg s} \cdot \norm{ \frac{1}{m} - (\bP^{t-s})_{i} }_1 + \frac{2L}{\reg t}.
		\label{eq:devBnd}
	\end{align}
	
We handle the term $\norm{ \frac{1}{m} - (\bP^{t-s})_{i} }_1$ using a bound on the mixing rate of Markov chains (c.f. (74) in Tsianos and Rabbat~\cite{DistStronglyConvex}):
	\begin{align}
	\sum_{s=1}^{t-1} \frac{L}{\reg s} \cdot \norm{ \frac{1}{m} - (\bP^{t-s})_{i} }_1 
		&\le \frac{L \sqrt{m}}{\reg}  \sum_{s=1}^{t-1} \left( \sqrt{ \lambda_2(\bP)} \right)^{t - s} \frac{1}{s}.
	\label{eq:mixingrate}
	\end{align}

Define $a = \sqrt{\lambda_2(\bP)} \le 1$ and $b = - \log(a) > 0$.  Then we have the following identities:
	\begin{align}
	\sum_{\tau=1}^t \frac{a^{t-\tau+1}}{\tau} 
	= \sum_{\tau=1}^t \frac{a^\tau}{t-\tau+1}
	=  \sum_{\tau=1}^t \frac{\exp(-b\tau)}{t-\tau+1}.
	\end{align}
Now using the fact that when $x > -1$ we have $\exp(-x)< 1/(1+x)$ and using the integral upper bound we get
	\begin{align}
   &	\sum_{\tau=1}^t \frac{a^{t-\tau+1}}{\tau} \notag \\
   	&\le
	\sum_{\tau=1}^t \frac{1}{(1+b\tau)(t-\tau+1)} \notag \\ 
	&\le  
	\frac{1}{(1+b)t} + \int_{1}^{t} \frac{d\tau}{(1+b\tau)(t-\tau+1)} \notag \\
 	&=  
	\frac{1}{(1+b)t} + \left[ \frac{\log(b\tau+1)
		-\log(t-\tau+1)}{bt+b+1} \right]_{\tau=1}^{t} 
		\notag \\
 	&=  
	\frac{1}{(1+b)t} + \frac{\log(bt+1)-\log(b+1)+\log(t)}{bt+b+1} \notag \\
 	&\le   
	\frac{\log(et(bt+1))}{bt} \notag \\
 	&\le   \frac{\log(2bet^2)}{bt}.
	\label{eq:seriesBnd}
	\end{align}
Using \eqref{eq:mixingrate} and \eqref{eq:seriesBnd} in \eqref{eq:devBnd} we get
	\begin{align}
	\norm{\bar{\w}(t)-\w_i(t)} 
	&\le \frac{L\sqrt{m}}{\reg}\frac{\log(2bet^2)}{bt} + \frac{2L}{ \reg t} 
	\notag \\
	&\le \frac{2L\sqrt{m}}{\reg}\frac{\log(2bet^2)}{bt}.
	\label{eq:networkDevBnd}
	\end{align}
Therefore we have %
	\begin{align}
	\sqrt{\E\left[\norm{ \bar{\w}(t)-\w_i(t) }^2 \right]}
		&\le \frac{2L\sqrt{m}}{\reg}\frac{\log(2bet^2)}{bt}.
	\label{eq:avgtoind:normbnd1}
\end{align}

\subsection{Bounds for expected gradient norms \label{subsec:gradBnds}}

\subsubsection{Bounding $\E\left[\norm{\nabla J_i(\bar{\w}(t))}^2\right]$}

Let $\sgi{j}{t} \in \partial \ell(\bar{\w}(t)^{\trans}\x_{i,j})$ denote a subgradient for the $j$-th point at node $i$ and $\SGI{t} = (\sgi{1}{t}, \sgi{2}{t}, \ldots, \sgi{n}{t})^{\trans}$ be the vector of subgradients at time $t$.  Let $\Q_{S_i}$ be the $n \times n$ Gram matrix of the data set $S_i$.  From the definition of $\norm{\nabla J_i(\bar{\w}(t))}$ and using the Lipschitz property of the loss functions, we have the following bound:
	\begin{align}
   	&\norm{\nabla J_i(\bar{\w}(t))}^2 
	\notag \\
	&\le \norm{\sum_{j \in S_i }\frac{\sgi{j}{t}\x_{i,j}}{n} + \reg \bar{\w}(t)}^2 
		\notag \\
	&\le
	2 \norm{\sum_{j\in S_i } \frac{\sgi{j}{t} \x_{i,j}}{n}}^2 
		+ 2\reg^2 \norm{\bar{\w}(t)}^2 
		\notag \\
	&= \frac{ 2 \sum_{j \in S_i }\sum_{j^{\prime}\in S_i } 
		\sgi{j}{t} \sgi{j^{\prime}}{t} \x_{i,j}^{\trans} \x_{i,j}^{\prime}  
		}{n^2} 
		+ 2\reg^2 \norm{\bar{\w}(t)}^2 \notag \\
&= \frac{ 2 }{n^2} \SGI{t}^{\trans} \Q_{S_i} \SGI{t}  + 2\reg^2 \norm{\bar{\w}(t)}^2 \notag \\
&\le \frac{2}{n^2} \norm{ \SGI{t} }^2 \spec(\Q_{S_i}) + 2\reg^2 \norm{\bar{\w}(t)}^2 \notag \\
&\le  2L^2 \frac{\spec(\Q_{S_i})}{n} + 2\reg^2 \norm{\bar{\w}(t)}^2.
\label{eq:basicGradBnd}
\end{align}

We rewrite the update \eqref{eq:avgUp} in terms of $\{ \x_{i,t} \}$, the points sampled at the nodes at time $t$:
	\begin{align}
	\bar{\w}(t+1) 
	= \bar{\w}(t)(1-\reg \step_t) 
		- \step_t \sum_{i=1}^{m} 
			\frac{\partial \ell(\w_i(t)^{\trans}\x_{i,t})\x_{i,t}}{m}.
	\label{eq:avgUprecurse}
	\end{align}
Now from equation \eqref{eq:avgUprecurse}, after unrolling the recursion as in Shalev-Shwarz et al.~\cite{SSSC11:pegasos} we see
	\begin{align}
	\bar{\w}(t) 
	= \frac{1}{\reg(t-1)} \sum_{\tau=1}^{t-1} 
		\frac{
		\sum_{i=1}^{m}\partial \ell(\w_i(\tau)^{\trans}\x_{i,\tau})\x_{i,\tau}
		}{m}.
\end{align}
Let $\gamma^i_{\tau} \in \partial \ell(\w_i(\tau)^{\trans}\x_{i,\tau})$ the subgradient set for the $i$th node computed at time $\tau$, then we have
	\begin{align}
	\norm{\bar{\w}(t)} 
	\le 
	\frac{1}{\reg(t-1)} \cdot \frac{1}{m} \sum_{i=1}^{m}  
		\norm{\sum_{\tau=1}^{t-1}\gamma^i_{\tau} \x_{i,\tau}}.
\label{eq:avg:normbnd1}
\end{align}

Let us in turn bound for each node $i$ the term $\norm{\sum_{\tau=1}^{t-1}\gamma^i_{\tau} \x_{i,\tau}}$. Let $\sgt{\tau}^{i} \in \partial \ell(\w_i(\tau)^{\trans}\x_{i,{\tau}})$ denote a subgradient for the point sampled at time $\tau$ at node $i$ and $\SGT^{i} = (\sgt{1}^{i}, \sgt{2}^{i}, \ldots, \sgt{t-1}^{i})^{\trans}$ be the vector of subgradients up to time $t-1$.  We have
	\begin{align}
	\norm{\sum_{\tau=1}^{t-1}\gamma^i_{\tau} \x_{i,\tau}}^2 
	&= \sum_{\tau,\tau^{\prime}} 
		\gamma^i_{\tau} \gamma^i_{\tau^{\prime}} 
		\x_{i,\tau}^{\trans}\x_{i,\tau^{\prime}} 
		\notag \\
	&= ( \mbs{\gamma}^{i} )^{\trans} \Q_{i,t-1} \mbs{\gamma}^{i} 
		\notag \\
	&\le \norm{\mbs{\gamma}^{i}}^2 \spec(\Q_{i,t-1}) 
		\notag \\
	&\le (t-1) L^2\spec(\Q_{i,t-1}),
\end{align}
where $\Q_{i,t-1}$ is the $(t-1) \times (t-1)$ Gram submatrix  corresponding to the points sampled at the $i$-th node until time $t-1$.

Further bounding \eqref{eq:avg:normbnd1}:
	\begin{align*}
	\norm{\bar{\w}(t)}^2 
	&\le \left(\frac{1}{\reg(t-1)} 
		\frac{\sum_{i=1}^{m} \sqrt{(t - 1) L^2\spec(\Q_{i,t-1})}}{m}\right)^2  
		\notag \\
	&\le \frac{L^2}{\reg^2} \left( \frac{1}{m} 
		\sum_{i=1}^{m} \sqrt{ \frac{\spec(\Q_{i,t-1})}{t-1} }
		\right)^2.
	\end{align*}
Since as stated before everything is conditioned on the sample split we take expectations w.r.t the history and the random split and using the Cauchy-Schwarz inequality again, and the fact that the points are sampled i.i.d. from the same distribution,
	\begin{align}
	&\E\left[ \norm{\bar{\w}(t)}^2 \right] 
	\notag \\
	&\le 
	\frac{L^2}{\reg^2} \frac{1}{m^2} \sum_{i=1}^{m} \sum_{j=1}^{m}
		\E\left[  \frac{ \sqrt{\spec(\Q_{i,t-1})\spec(\Q_{j,t-1}) } }{t-1}
			\right] \notag \\
	&\le \frac{L^2}{\reg^2} \frac{1}{m^2}
		\sum_{i=1}^{m} \sum_{j=1}^{m}
		\sqrt{ \E\left[ \frac{ \spec(\Q_{i,t-1})  }{t-1} \right]
			\E\left[ \frac{ \spec(\Q_{j,t-1}) }{t-1} \right] } \notag \\
	&= \frac{L^2}{\reg^2} \E\left[ \frac{ \spec(\Q_{i,t-1})  }{t-1} \right] .
	\label{eq:avg:normbnd2}
	\end{align}
The last line follows from the expectation over the sampling model: the data at node $i$ and node $j$ have the same expected covariance since they are sampled uniformly at random from the total data.

Taking the expectation in \eqref{eq:basicGradBnd} and substituting \eqref{eq:avg:normbnd2} we have
	\begin{align}
   	\E\left[\norm{\nabla J_i(\bar{\w}(t))}^2\right] 
	&\le 
		2L^2 \E\left[\frac{\spec(\Q_{S_i})}{n}\right] 
		\notag \\
		&\hspace{0.5in}
		+ 2L^2 \E\left[\frac{\spec(\Q_{i,t-1})}{t-1} \right].
	\end{align}
Since $S_i$ is a uniform random draw from $S$ and by assuming both $t$ and $n$ to be greater than $4/(3\specnorm^2)\log (d)$, applying Lemma \ref{lem:specnormIntdim} gives us
	\begin{align}
	\E\left[\norm{\nabla J_i(\bar{\w}(t))}^2\right] \le 20L^2\specnorm^2.
	\label{eq:itravgnrmBnd}
	\end{align}

\subsubsection{Bounding $\E\left[\norm{\nabla J_i(\w_i(t))}^2\right]$}

We have just as in the previous subsection
\begin{align*}
\norm{\nabla J_i(\w_i(t))}^2 &\le  2L^2 \frac{\spec(\Q_{S_i})}{n} + 2\reg^2 \norm{\w_i(t)}^2.
\end{align*}
Using the triangle inequality, the fact that $(a_1+a_2)^2 \le 2a_1^2 + 2a_2^2$, the bounds \eqref{eq:networkDevBnd} and \eqref{eq:avg:normbnd2}, and Lemma \ref{lem:specnormIntdim}:
\begin{align}
\E\left[\norm{\w_i(t)}^2\right] &\le 2\E\left[\norm{\w_i(t) - \bar{\w}(t)}^2\right]  + 2\E\left[\norm{\bar{\w}(t)}^2\right] \notag \\
&\le  \frac{8L^2m}{\reg^2}\frac{\log^2(2bet^2)}{b^2 (t-1)^2} + \frac{5L^2\specnorm^2}{\reg^2}.
\label{eq:netBnd}
\end{align}
Since the second term does not scale with $t$, from \eqref{eq:netBnd} we can infer that for the second term to dominate the first we require
\begin{align*}
\frac{t}{\log(t)} > \sqrt{ \frac{8}{5} } \frac{\sqrt{m}}{\specnorm b}.
\end{align*}
This gives us
\begin{align}
\E\left[\norm{\w_i(t)}^2\right] &\le \frac{10L^2\specnorm^2}{\reg^2},
	\label{eq:localiter:norm}
\end{align}
and therefore
\begin{align}
\E\left[ \norm{\nabla J_i(\w_i(t))}^2 \right] \le 30L^2\specnorm^2.
\label{eq:itrnrmBnd}
\end{align}

\subsection{Bound for $T2$ }

Because the gradients are bounded,
\begin{align*}
&\E \left[ \norm{\sum_{i=1}^{m}\frac{\g_i(t)}{m}}^2 \right] 
\notag \\
&=\E \left[\sum_{i,j}\frac{\g_i(t)^{\trans}\g_i(t)}{m^2} \right] \notag \\
&=\sum_{i=1}^m \frac{\E \left[ \norm{\g_i(t)}^2 \right]}{m^2} + \sum_{i\ne j}\frac{\E \left[\g_i(t)^{\trans}\g_j(t)\right]}{m^2}  \\
&\le \frac{L^2}{m} +  \sum_{i\ne j}\frac{\E \left[\g_i(t)^{\trans}\g_j(t)\right]}{m^2}  \\
&= \frac{L^2}{m} +\frac{\sum_{i \ne j} \E_{\calF_{t-1}} \left[ \E\left[\g_i(t)^{\trans}\g_j(t) |  \calF_{t-1} \right] \right]}{m^2}.
\end{align*}
Now using the fact that the gradients $\g_i(t)$ are unbiased estimates of $\nabla J_i(\w_t)$ and that $\g_i(t)$ and $\g_j(t)$ are independent given past history and inequality \eqref{eq:itrnrmBnd} for node $i$ and $j$ we get
\begin{align}
&\frac{\sum_{i \ne j} \E_{\calF_{t-1}} \left[ \E\left[\g_i(t)^{\trans}\g_j(t) |  \calF_{t-1} \right] \right]}{m^2} 
\notag \\ 
&=\sum_{i \ne j} \frac{\E_{\calF_{t-1}} \left[\nabla J_i(\w_i(t))^{\trans}\nabla J_j(\w_j(t))  \right]}{m^2} \notag \\
&\le \sum_{i \ne j} \frac{\sqrt{\E_{\calF_{t-1}} \left[\norm{\nabla J_i(\w_i(t))}^2\right]}\sqrt{\E_{\calF_{t-1}} \left[\norm{\nabla J_j(\w_j(t))}^2\right]}}{m^2} \notag \\
&
= \frac{(m-1)}{m} \cdot 30L^2\specnorm^2\notag \\
&
\le 30L^2\specnorm^2.
\end{align}
Therefore to bound the term $\mathrm{T2}$ in \eqref{eq:mainBnd} we can use
\begin{align}
\E \left[ \norm{\sum_{i=1}^{m}\frac{\g_i(t)}{m}}^2 \right] 
	\le \frac{L^2}{m} + 30L^2\specnorm^2.
\label{eq:T2Bnd}
\end{align}

\subsection{Bound for $T3$ }

Applying \eqref{eq:avgtoind:normbnd1}, \eqref{eq:itravgnrmBnd}, and \eqref{eq:itrnrmBnd} to $\mathrm{T3}$ in \eqref{eq:mainBnd}, as well as Lemma \ref{lemma:avgdevBnd} and the fact that $(a_1+a_2)^2\le 2a_1^2 + 2a_2^2$ we obtain the following bound:
\begin{align}
T3 &\le \frac{1}{m} \sum_{i=1}^{m} 
	\sqrt{\E\left[\norm{ \bar{\w}(t)-\w_i(t) }^2 \right]} 
	\notag \\
	&\hspace{0.5in}
	\cdot \sqrt{\E\left[ \left(\norm{\nabla J_i(\w_i(t))} 
				+ \norm{\nabla J_i(\bar{\w}(t))} \right)^2 \right] }
	\notag \\
 &\le \frac{1}{m} \sum_{i=1}^{m} \frac{2L\sqrt{m}}{\mu}\frac{\log(2bet^2)}{bt} \cdot 10L \specnorm \notag \\
 &\le   \frac{20L^2}{\mu} \cdot \frac{\sqrt{m}}{b} \cdot \frac{\log(T)}{t} \cdot \specnorm.
\label{eq:T3Bnd}
\end{align}

\subsection{Combining the Bounds \label{sec:combo_bound}}

Finally combining \eqref{eq:T2Bnd} and \eqref{eq:T3Bnd} in \eqref{eq:mainBnd} and applying the step size assumption $\step_t = 1/(\reg t)$:
\begin{align}
&\E \left[ J(\bar{\w}(t))- J(\w^{*}) \right]  
\notag \\
&\le 
	\frac{(\step_t^{-1}-\reg)}{2}\E \left[ \norm{\bar{\w}(t)-\w^{*}}^2 \right] 
	\notag \\
	&\qquad 
	- \frac{\step_t^{-1}}{2}\E\left[ \norm{\bar{\w}(t+1)-\w^{*}}^2 \right]
	\notag \\
	&\qquad 
	+  \left(\frac{30L^2\specnorm^2}{\reg} + \frac{L^2}{\reg m}\right)\cdot \frac{1}	{t}\notag \\
	&\qquad 
	+  \frac{20L^2}{\reg} \cdot \frac{\sqrt{m}}{b} \cdot \frac{\log(2bet^2)}{t} \cdot \specnorm 
\notag \\
&\le 
	\frac{\reg(t-1)}{2}\E \left[ \norm{\bar{\w}(t)-\w^{*}}^2 \right]  
	\notag \\
	&\qquad
	- \frac{\reg t}{2}\E\left[ \norm{\bar{\w}(t+1)-\w^{*}}^2 \right] 
	+ K_0 \cdot \frac{L^2}{\reg t},
 \label{eq:mainBnd2}
 \end{align}
where $K_0 =  \left(30 \specnorm^2 + 1/m + \left(60 \cdot \sqrt{m\specnorm^2} \cdot \log(T)\right)/b \right)$, using $t\le T$ and assuming $T>2be$.

Let us now define two new sequences, the average of the average of iterates over nodes from $t=1$ to $T$ and the average for any node $i\in [m]$ 
\begin{align}
\hat{\w}(T) &= \frac{1}{T}  \sum_{t=1}^T \bar{\w}(t) \\
\hat{\w}_i(T) &= \frac{1}{T} \sum_{t=1}^T \w_i(t).
\end{align} 
Then summing \eqref{eq:mainBnd2} from $t=1$ to $T$, using the convexity of $J$ and collapsing the telescoping sum in the first two terms of \eqref{eq:mainBnd2},
	\begin{align}
&\E \left[ J(\hat{\w}(T))- J(\w^{*}) \right]  \notag \\
&\le \frac{1}{T} \sum_{t=1}^T \E \left[ J(\bar{\w}(t))- J(\w^{*}) \right] 
	 \notag \\
&\le - \frac{\reg T}{2} \E\left[ \norm{\bar{\w}(T+1)-\w^{*}}^2 \right] + K_0 \cdot \frac{L^2}{\reg} \cdot \frac{\sum_{t=1}^{T} 1/t}{T}  	\notag \\
&\le  K_0 \cdot \frac{L^2}{\reg} \cdot \frac{\log(T)}{T}.
\label{eq:avgofavgBnd} 
\end{align}

Now using the definition of subgradient, Cauchy-Schwarz, and Jensen's inequality we have 
\begin{align}
	&J(\hat{\w}_i(T)) -J(\w^{*}) \notag \\
   &\le	J(\hat{\w}(T))-J(\w^{*}) + \nabla J(\hat{\w}_i(T))^{\trans}(\hat{\w}_i(t) - \hat{\w}(T)) 
	\notag \\
	&\le
	J(\hat{\w}(T))-J(\w^{*}) + \norm{\nabla J(\hat{\w}_i(T)}\norm{\hat{\w}_i(t) - \hat{\w}(T)} 
	\notag \\
	&\le
	J(\hat{\w}(T))-J(\w^{*})  
	\notag \\
	&\hspace{0.5in}
	+ \norm{\nabla J(\hat{\w}_i(T))} \cdot \sum_{t=1}^T \frac{\norm{\w_i(t)-\bar{\w}(t)}}{T}.
	\label{eq:combine:bound1}
\end{align}
To proceed we must bound $\E\left[ \norm{\nabla J(\hat{\w}_i(T))}^2 \right]$ in a similar way as the bound \eqref{eq:itravgnrmBnd}.  First, let $\alpha_{i} = \partial \ell( \hat{\w}_i(T)^{\trans} \x_i)$ denote the subgradient for the $i$-th loss function of $J(\cdot)$ in \eqref{eq:optForm}, evaluated at $\hat{\w}_i(T)$, and $\mbs{\alpha}_T = (\alpha_1, \alpha_2, \ldots, \alpha_N)^{\trans}$ be the vector of subgradients.  As before,
	\begin{align*}
	\norm{ \nabla J(\hat{\w}_i(T)) }^2
	&= \norm{ \frac{1}{N} \sum_{i=1}^{N} \alpha_i \x_i + \reg \hat{\w}_i(T) }^2 \\
	&\le \frac{2}{N^2} \mbs{\alpha}^{\trans} \Q \mbs{\alpha} + 2 \reg^2 \norm{ \hat{\w}_i(T) }^2 \\
	&\le 10 L^2 \specnorm^2 + 2 \reg^2 \norm{ \hat{\w}_i(T) }^2 \\
	&\le 10 L^2 \specnorm^2 + 2 \reg^2 \frac{1}{T} \sum_{t=1}^{T} \norm{ \w_i(t) }^2. 
	\end{align*}
Taking expectations of both sides and using \eqref{eq:localiter:norm} as before:
	\begin{align*}
	\E\left[ \norm{ \nabla J(\hat{\w}_i(T)) }^2 \right]
	&\le 30 L^2 \specnorm^2 .
	\end{align*}
Taking expectations of both sides of \eqref{eq:combine:bound1} and using the Cauchy-Schwarz inequality, \eqref{eq:avgofavgBnd}, the preceding gradient bound, Lemma \ref{lemma:avgdevBnd} and the definition of $K_0$ we get
	\begin{align}
	&\E\left[ J(\hat{\w}_i(T)) -J(\w^{*})  \right] \notag \\
	&\le
	K_0 \cdot \frac{ L^2}{\reg} \cdot \frac{\log (T)}{T}
		+ \frac{2 \sqrt{30} L^2}{\reg} \cdot \frac{\sqrt{m}}{b} \cdot \specnorm \cdot \frac{\log(T)}{T} \cdot \sum_{t=1}^T\frac{1}{t}	\notag \\
	&\le \left(K_0 + \frac{2\sqrt{30}\cdot \sqrt{m\specnorm^2} \cdot \log T}{b}  \right)\cdot \frac{\log T}{T} \notag \\
	&\le \left(30\specnorm^2 + \frac{1}{m} + \frac{70\sqrt{m\specnorm^2}\cdot \log T}{b} \right) 
	\cdot \frac{L^2}{\reg} \cdot \frac{\log T}{T}.
	\end{align}
Recalling that $b=\log(1/\lambda_2(\bP)) \ge 1-\lambda_2(P)$, assuming $T> 2be$ and subsuming the first term in the third and taking expectations with respect to the sample split the above bound can be written as
\begin{align}
\E\left[ J(\hat{\w}_i(T)) -J(\w^{*})  \right] 
&\le \left(\frac{1}{m} + \frac{100\sqrt{m\specnorm^2}\cdot \log T}{1-\lambda_2(P)} \right)
\notag \\
&\hspace{0.5in}
\cdot \frac{L^2}{\reg} \cdot \frac{\log T}{T}.
\end{align}
\end{proof}

\section{Proof of Lemma \ref{lemma:avgdevBndStoch}}

\begin{proof}
Let us define the product of the sequence of random matrices $\{\bP(\tau) : s \le \tau \le t \}$:
\begin{align}
\Phi(s:t) = \bP(t) \cdots \bP(s).
\end{align}
Then proceeding as in proof of Lemma \ref{lemma:avgdevBnd} and using the step size $\step_t = 1/(\reg t)$, we get
\begin{align}
	\norm{ \bar{\w}(t)-\w_i(t) } 
	&\le
		\Bigg\| 
			\sum_{s=1}^{t-1} \step_s \sum_{j=1}^{m} 
				\left( \frac{1}{m} -  \Phi(s:t)_{ij} \right) \g_j(s) 
			\notag \\
			&\hspace{0.4in}
			+ \step_t \left( \sum_{j=1}^{m} \frac{1}{m} \g_j(t) - \g_i(t) \right)  
			\Bigg\| \\
	&\le \sum_{s=1}^{t-1} \frac{L}{\reg s} \cdot \norm{ \frac{1}{m} -  \Phi(s:t)\mathbf{e}_i }_1 + \frac{2L}{\mu t}.
		\label{eq:devBnd2}
	\end{align}

Let $\mathbf{e}_i$ be a vector with $0$'s everywhere except at the the $i$th position, then
\begin{align}
\E \left[\norm{\frac{\mathsf{1}}{m} - \Phi(s:t)\mathbf{e}_i}_1 \right] &\leq \sqrt{m} \E \left[\norm{\frac{\mathsf{1}}{m} - \Phi(s:t)\mathbf{e}_i}_2 \right].
\end{align}
Consider the recursion $\mathbf{u}(t+1) = \bP(t)\mathbf{u}(t)$ and let $\mathbf{v}(t+1) = \bP(t)\mathbf{u}(t) - \frac{\mathsf{1}}{m}$ then we have
\begin{align}
\E\left[\mathbf{v}(t+1)^{\trans}\mathbf{v}(t+1) | \mathbf{v}(t)\right] &= \E\left[ \mathbf{v}(t)^{\trans}\bP^2(t)\mathbf{v}(t) | \mathbf{v}(t)\right] \notag \\
&=  \mathbf{v}(t)^{\trans}\E\left[\bP^2(t)\right]\mathbf{v}(t) \notag \\ 
&\leq \norm{\mathbf{v}(t)}^2\lambda_2\left(\E\left[\bP^2(t)\right]\right),
\end{align}
since $\mathbf{v}(t)$ is orthogonal to the largest eigenvector of $\bP(t)$.

Taking expectations w.r.t to $\mathbf{v}(t)$ we get
\begin{align}
\E\left[\norm{\mathbf{v}(t+1)}^2\right] &\leq \E\left[\norm{\mathbf{v}(t)}^2\right]\lambda_2\left(\E\left[\bP^2(t)\right]\right).
\label{eq:errrecursion}
\end{align}
Recursively expanding \eqref{eq:errrecursion} we obtain
\begin{align}
\E\left[\norm{\mathbf{v}(t+1)}^2\right] &\leq \norm{\mathbf{v}(0)}^2\lambda_2 \left(\E\left[\bP^2(t)\right]\right)^{t-s+1}.
\end{align}
Consider an initial vector $\mathbf{u}(0)=\mathbf{e}_i$.  We see that $\norm{\mathbf{v}(t+1)}^2 = \norm{\frac{\mathsf{1}}{m} - \Phi(s:t)_i}_2$, this finally gives us
\begin{align}
&\E \left[\norm{\frac{\mathsf{1}}{m} - \Phi(s:t)_i}_1 \right] 
\notag \\
&\leq 
\sqrt{m} \E \left[\norm{\frac{\mathsf{1}}{m} - \Phi(s:t)_i}_2 \right]  \notag \\
&\leq\sqrt{m} \norm{\mathbf{e}_i - \frac{\mathsf{1}}{m}}^2 \lambda_2\left(\E\left[\bP^2(t)\right]\right)^{t-s+1} \notag \\
&\leq \sqrt{m} \lambda_2\left(\E\left[\bP^2(t)\right]\right)^{t-s+1}.
\end{align}
Proceeding like the proof of Lemma \ref{lemma:avgdevBndStoch} where $a=\lambda_2\left(\E\left[\bP^2(t)\right]\right)$ and $b=-\log(a)$ we get 
	\begin{align}
	\sqrt{\E\left[\norm{ \bar{\w}(t)-\w_i(t) }^2 \right]}
		&\le \frac{2L\sqrt{m}}{\reg}\frac{\log(2bet^2)}{bt}.
\end{align}

\end{proof}

\section{Proof of Theorem \ref{theorem:mainThrmStoch}}

The proof follows easily from the proof of Theorem \ref{theorem:mainThrm}.

\begin{proof}
Since \eqref{eq:mainBnd} still holds, we merely apply Lemma \ref{lemma:avgdevBndStoch} in \eqref{eq:mainBnd} and continue in the same way as the proof of Theorem \ref{theorem:mainThrm}.
\end{proof}

\section{Proof of Theorem \ref{theorem:mainThrm:mbatch}}

We will first establish the network lemma for scheme \eqref{eq:mbatchComm}.

\begin{lemma}\label{lemma:avgdevBndmbatch}
Fix a Markov matrix $\bP$ and consider Algorithm \ref{alg:DiSCO} when the objective $J(\w)$ is strongly convex and the frequency of communication satisfies
    \begin{align}
	1/\nu > \frac{4}{3\specnorm^2} \log(d) 
	\end{align}
we have the following inequality for the expected squared error between the iterate $\w_i(t)$ at node $i$ at time $t$ and the average $\bar{\w}(t)$ defined in Algorithm \ref{alg:DiSCO} for scheme \eqref{eq:mbatchComm}
	\begin{align}
	\sqrt{\E\left[\norm{ \bar{\w}(t)-\w_i(t) }^2 \right]} 
	\leq \frac{4L\sqrt{5m \specnorm^2}}{\reg}\cdot \frac{\log(2bet^2)}{bt}
	\end{align}
where $b = (1/2)\log(1/\lambda_2(\bP)) $.
\end{lemma}

\begin{proof}
It is easy to see that we can write the update equation in Algorithm \ref{alg:DiSCO} as
\begin{align}
\w_i(t+1) &= \sum_{j=1}^m\tilde{\bP}_{ij}(t)\w_j(t)- \step_t \g^{1/\nu}_i(t)
\end{align}
where
\begin{align}
\tilde{\bP}_{ij}(t) =
\left\{
	\begin{array}{ll}
		P_{ij}(t)  &  \text{ when } i \ne j \\
		P_{ii}(t)-\frac{1}{mt} &  \text{ when } i=j
	\end{array}
\right.
\end{align}
and $\g_i(t) = \g^{1/\nu}_i(t) + \reg\w_i(t)$.

We need first a bound on $\norm{\g^{1/\nu}_j(s)}$ using the definition of the minibatch (sub)gradient:
 \begin{align}
\norm{\g^{1/\nu}_i(s)}^2 &= \norm{ \frac{\sum_{i_{k_s} \in H^i_s}\partial \ell(\w_i(s)^{\trans}\x_{k_{i_s}})\x_{k_{i_s}} }{1/\nu}}^2  \notag \\
&\leq L^2\nu \norm{\Q_{1/\nu}} 
\end{align} 
From \eqref{eq:devBnd} and the minibatch (sub)gradient bound
	\begin{align*}
   	&
   	\norm{ \bar{\w}(t)-\w_i(t) } \notag \\
	&\le 
	\norm{ \sum_{s=1}^{t-1} \step_s \sum_{j=1}^m \left( \frac{1}{m} - (\tilde{\bP}^{t-s})_{ij} \right) \g^{1/\nu}_j(s)} 
	\notag \\ 
  		&\qquad
		+ \step_t \norm{\left( \sum_{j=1}^{m} \frac{\mbf{1}}{m} \g^{1/\nu}_j(t) - \g^{1/\nu}_i(t) \right) } \\
	&\le L\sqrt{\nu\norm{\Q_{1/\nu}}}\sum_{s=1}^{t-1} \frac{\norm{ \frac{\mbf{1}}{m} - (\tilde{\bP}^{t-s})_{i} }_1}{\reg s} + \frac{2L\sqrt{\nu\norm{\Q_{1/\nu}}}}{\reg t} \\
	&\le L\sqrt{\nu\norm{\Q_{1/\nu}}}
		\notag \\
		&\hspace{0.6in}
		\sum_{s=1}^{t-1} \frac{\norm{ \frac{\mbf{1}}{m} - (\bP^{t-s})_{i} }_1 + \norm{ (\bP^{t-s})_{i} - (\tilde{\bP}^{t-s})_{i} }_1}{\reg s} \notag \\
		&\qquad + \frac{2L\sqrt{\nu\norm{\Q_{1/\nu}}}}{\reg t} \\
	&\le 
	2L\sqrt{\nu\norm{\Q_{1/\nu}}}\sum_{s=1}^{t-1} \frac{\norm{ \frac{\mbf{1}}{m} - (\bP^{t-s})_{i} }_1}{\reg s} + \frac{2L\sqrt{\nu\norm{\Q_{1/\nu}}}}{\reg t} 
	\end{align*}
	
Continuing as in the proof of Lemma \ref{lemma:avgdevBnd}, taking expectations and using Lemma \ref{lem:specnormIntdim}, for 
$1/\nu > \frac{4}{3\specnorm^2} \log(d)$ we have
	\begin{align}
	\sqrt{\E\left [\norm{ \bar{\w}(t)-\w_i(t) }^2 \right ]}
		&\le \frac{4L\sqrt{m\nu\E\left[\norm{\Q^{1/\nu}}\right]}}{\reg}\frac{\log(2bet^2)}{bt} \notag \\
		&\le \frac{4L\sqrt{5m\specnorm^2}}{\reg}\frac{\log(2bet^2)}{bt}
		\label{eq:avgtoind:normbnd:minibatch}
\end{align}

\end{proof}

For the scheme \eqref{eq:mbatchComm} all the steps until bound \eqref{eq:mainBnd} from proof of Theorem \ref{theorem:mainThrmStoch} remain the same. The difference in the rest of the proof arises primarily from the mini batch gradient norm factor in Lemma \ref{lemma:avgdevBndmbatch}. We have the same decomposition as \eqref{eq:mainBnd} with $\text{T1}$, $\text{T2}$, and $\text{T3}$ as in \eqref{eq:T1}, \eqref{eq:T2}, and \eqref{eq:T3}.
 The gradient norm bounds also don't change since the minibatch gradient is also an unbiased gradient of the true gradient $\nabla J(\cdot)$. Thus substituting Lemma \ref{lemma:avgdevBndmbatch} in the above and following the same steps as in proof of Theorem \ref{theorem:mainThrmStoch}, replacing $T$ by $\nu T$ where $T$ is now the total iterations including the communication as well as the minibatch gathering rounds, we get Theorem \ref{theorem:mainThrm:mbatch}.
 
 \subsection{Proof of Lemma \ref{lemma:mainThrmLarge}}

In the proof we will \removed{first show that the iterate of Algorithm \ref{alg:DiSCO} is asymptotically normal by showing it is close to the iterate of distributed algorithm of \cite{BianchiFortHachem:13IEEETrans}  and then} use the corresponding multivariate normality result of Bianchi et al.~\cite[Theorem 5]{BianchiFortHachem:13IEEETrans}. Finally using smoothness and strong convexity we shall get Lemma \ref{lemma:mainThrmLarge}.  

It is easy to verify that Algorithm \ref{alg:DiSCO} satisfies all the assumptions necessary (Assumptions \textbf{1}, \textbf{4}, \textbf{6}, \textbf{7}, \textbf{8a}, and \textbf{8b} in Bianchi et al.~\cite{BianchiFortHachem:13IEEETrans}) for the result to hold.
\begin{itemize}
\item Assumption \textbf{1} requires the weight matrix $\bP(t)$ to be row stochastic almost surely, identically distributed over time, and that $\E[\bP(t)]$ is column stochastic. Our Markov matrix is constant over time and doubly stochastic. Assumption \textbf{1b} follows because $\bP$ is constant and independent of the stochastic gradients, which are sampled uniformly with replacement.
\item Assumption \textbf{4} requires square integrability of the gradients as well as a regularity condition. In our setting, this follows since the sampled gradients are bounded almost everywhere.
\item Assumption \textbf{6} imposes some analytic conditions at the optimum value.  These hold since the gradient is assumed to be differentiable and the Hessian matrix at $\w^{*}$ is positive definite with its smallest eigenvalue is at least $\reg > 0$ (this follows from strong convexity). 
\item Assumption \textbf{7} of Bianchi et al.~\cite{BianchiFortHachem:13IEEETrans} follows from our existing assumptions.
\item Assumptions \textbf{8a} and \textbf{8b} are standard stochastic approximation assumptions on the step size that are easily satisfied by $\step_t = \frac{1}{\reg t}$.
\end{itemize} 

Next it is straightforward to show that the average over the nodes of the iterates $\tilde{\w}_i(t)$, $\w_i(t)$ for Algorithm \ref{alg:DiSCO} and distributed algorithm of ~\cite{BianchiFortHachem:13IEEETrans} are the same and satisfy
\begin{align}
\bar{\tilde{\w}}(t+1) = \bar{\tilde{\w}}(t) - \eta_t \frac{\sum_{i=1}^m \g_i(t)}{m}  \notag \\ 
\bar{\w}_i(t+1) = \bar{\w}_i(t+1) - \eta_t \frac{\sum_{i=1}^m \g_i(t)}{m}  
\label{eq:avgIterates}
\end{align}

Now note that 
\begin{align}
\w_i(t) - \w^{*} = \underbrace{\w_i(t) - \bar{\w}_i(t)}_{\textbf{T1=Network Error}} + \underbrace{\bar{\w}_i(t) - \w^{*}}_{\textbf{T2=Asymptotically Normal}}
\label{eq:NetworkErrAsymNormal}
\end{align}

From Lemma \ref{lemma:avgdevBnd} we know that the network error (\textbf{T1}) decays and from update equation \eqref{eq:avgIterates} we know that the averaged iterates for both the versions are the same . Then the proof of Theorem $5$ of Bianchi et al.~\cite{BianchiFortHachem:13IEEETrans} shows that the term \textbf{T2}, under the above assumptions when appropriately normalized converges to a centered Gaussian distribution. Equation \eqref{eq:NetworkErrAsymNormal} then implies
\begin{align}
\sqrt{\reg t}\left(\w_i(t) - \w^{*}\right) \sim \mc{N}\left(\mbf{0} , \mbf{H}\right ),
\label{eq:multinormDist}
\end{align}
where $\mbf{H}$ is the solution to the equation
\begin{align}
\nabla J^2(\w^{*}) \mbf{H} + \mbf{H} \nabla J^2(\w^{*})^{T} = \mbf{C}.
\end{align}
Let $\Y \sim \mc{N}(\mbf{0},\mc{\mbf{I}})$, so we can always write for any $\mbf{X} \sim \mc{N}(\mbf{0},\mbf{\mbf{H}})$
\begin{align}
\mbf{X} = \Y \mbf{H}^{1/2},
\end{align}
and thus
\begin{align}
\norm{\mbf{X}}^2 = \Y^{\trans}\mbf{H}\Y
\end{align}
Then it is well known that $\norm{\mbf{X}}^2 \sim \chi^2(\Tr(\mbf{H}))$ and so $\E\left[ \norm{\mbf{X}}^2 \right] = \Tr(\mbf{H})$.

Let us now consider the suboptimality at the iterate $\sum_{j=1}^m P_{ij} \w_j(t)$. It is easy to see that for a differentiable and strongly convex function
\begin{align}
J\left(\sum_{j=1}^m P_{ij} \w_j(t)\right) - J(\w^{*}) &\leq \frac{G}{2} \norm{\sum_{j=1}^m P_{ij} \w_j(t) - \w^{*}}^2.
\label{eq:suboptAsym}
\end{align}

Now it is easy to see from \eqref{eq:multinormDist} that for a node $j \in \mc{N}(i)$
\begin{align}
P_{ij}\sqrt{\reg t}\left(\w_j(t) - \w^{*}\right) \sim \mc{N}\left(\mbf{0} , (P_{ij})^2 \mbf{H}\right ).
\end{align}
This implies that
\begin{align}
\sum_{j \in \mc{N}(i)} P_{ij}\sqrt{\reg t}\left(\w_j(t) - \w^{*}\right) \sim \mc{N}\left(\mbf{0} ,  \left(\sum_{j \in \mc{N}(i)}(P_{ij})^2\right) \mbf{H} \right ).
\label{eq:asmypDist}
\end{align}
Then taking expectation w.r.t to the distribution \eqref{eq:asmypDist} and using standard properties of norms of multivariate normal variables,
\begin{align}
&\E\left[\norm{\sum_{j \in \mc{N}(i)} P_{ij}\sqrt{\reg t}\left(\w_j(t) - \w^{*}\right)}^2\right]  \notag \\
&=\left(\sum_{j \in \mc{N}(i)}(P_{ij})^2\right) \Tr\left(\mbf{H}\right).
\end{align}

Then substituting in bound \eqref{eq:suboptAsym} and taking the limit we finally get
\begin{align}
&\limsup\limits_{T\rightarrow \infty} T\cdot \E\left[J\left(\sum_{j=1}^m P_{ij} \w_j(T)\right) - J(\w^{*})\right] \notag \\
&\le \sum_{j \in \mc{N}(i)}(P_{ij})^2 \cdot \Tr\left(\mbf{H}\right) \cdot \frac{G}{\reg}.
\end{align}

\subsection{Proof of Theorem \ref{theorem:mainThrmLargeApplication}}

The the covariance of the gradient noise under the sampling with replacement model is
\begin{align}
\mathbf{C} &= \E\left[\g_i(t)\g_i(t)^{\trans}\right] - \nabla J(\w_i(t)) \nabla J(\w_i(t))^{\trans} \notag \\
  &= \frac{\sum_{i=1}^N \beta_{i,t}\x_i\x_i^{T}}{N} + \frac{\mu}{N}\sum_{i=1}^N \beta_{i,t} \left(\x_i\w_i(t)^{\trans} +  \w_i(t)\x_i^{\trans}\right) \notag \\
  &+ \mu^2 \w_i(t)\w_i(t)^{\trans} - \nabla J(\w_i(t)) \nabla J(\w_i(t))^{\trans}. \notag \\
\end{align}

Thus we can bound the spectral norm of $\mathbf{C}$ as
\begin{align}
\spec(\mathbf{C}) &\leq L^2 \rho^2 + 2\mu L \E\left[\norm{\w_i(t)}\right] + \mu^2 \E\left[\norm{\w_i(t)}^2\right] \notag \\
                 &+ \E\left[\norm{\nabla J(\w_i(t))}^2\right].
\end{align}

Now from bound \eqref{eq:localiter:norm} since $T \rightarrow \infty$ we have 
\begin{align*}
&\E\left[\norm{\w_i(t)}^2\right] \le \frac{10L^2\specnorm^2}{\reg^2} \notag \\ 
&\E\left[ \norm{\nabla J_i(\w_i(t))}^2 \right] \le 30L^2\specnorm^2.
\end{align*}

Putting everything together we get
\begin{align}
\spec(\mathbf{C}) \leq 50\rho L^2.
\label{eq:noiseCov}
\end{align}

Next note that $H = C\left(\nabla^2 J(\w^{*})\right)^{-1}/2$.  From the completeness and uniform weight assumptions on the graph, we have
\begin{align}
\sum_{j \in \mc{N}(i)}(P_{ij})^2 = \frac{1}{m}.
\end{align}

Thus substituting in Lemma \ref{lemma:mainThrmLarge}, using \eqref{eq:noiseCov} gives us 
\begin{align}
&\limsup\limits_{t\rightarrow \infty} t \cdot \E\left[J\left(\sum_{j=1}^m P_{ij} \w_j(t)\right) - J(\w^{*})\right] 
\notag \\
&\le \frac{1}{m} \cdot \frac{\Tr\left(\left(\mbf{C} \nabla^2 J(\w^{*})\right)^{-1} \right)}{2} \cdot \frac{G}{\reg} \notag \\
&\le \frac{25\rho L^2}{m} \cdot \Tr\left(\nabla^2 J(\w^{*})^{-1} \right) \cdot \frac{G}{\reg}. \notag
\label{eq:finalbndlarge}
\end{align}

\bibliographystyle{IEEEtran}
\bibliography{opt}

\begin{thebibliography}{10}
\providecommand{\url}[1]{#1}
\csname url@samestyle\endcsname
\providecommand{\newblock}{\relax}
\providecommand{\bibinfo}[2]{#2}
\providecommand{\BIBentrySTDinterwordspacing}{\spaceskip=0pt\relax}
\providecommand{\BIBentryALTinterwordstretchfactor}{4}
\providecommand{\BIBentryALTinterwordspacing}{\spaceskip=\fontdimen2\font plus
\BIBentryALTinterwordstretchfactor\fontdimen3\font minus
  \fontdimen4\font\relax}
\providecommand{\BIBforeignlanguage}[2]{{%
\expandafter\ifx\csname l@#1\endcsname\relax
\typeout{** WARNING: IEEEtran.bst: No hyphenation pattern has been}%
\typeout{** loaded for the language `#1'. Using the pattern for}%
\typeout{** the default language instead.}%
\else
\language=\csname l@#1\endcsname
\fi
#2}}
\providecommand{\BIBdecl}{\relax}
\BIBdecl

\bibitem{RakhShamir:12arxiv}
\BIBentryALTinterwordspacing
A.~Rakhlin, O.~Shamir, and K.~Sridharan, ``Making gradient descent optimal for
  strongly convex stochastic optimization,'' ArXiV, extended version of ICML
  paper arXiv:1109.5647 [cs.LG], 2012. [Online]. Available:
  \url{http://arxiv.org/abs/1109.5647}
\BIBentrySTDinterwordspacing

\bibitem{UMLbook}
S.~Shalev-Shwartz and S.~Ben-David, \emph{Understanding Machine Learning: From
  Theory to Algorithms}.\hskip 1em plus 0.5em minus 0.4em\relax Cambridge, UK:
  Cambridge, 2014.

\bibitem{nedicDistributedOptimization}
\BIBentryALTinterwordspacing
A.~Nedic and A.~Ozdaglar, ``Distributed subgradient methods for multi-agent
  optimization,'' \emph{IEEE Transactions on Automatic Control}, vol.~54,
  no.~1, pp. 48--61, January 2009. [Online]. Available:
  \url{http://dx.doi.org/10.1109/TAC.2008.2009515}
\BIBentrySTDinterwordspacing

\bibitem{SSSC11:pegasos}
\BIBentryALTinterwordspacing
S.~Shalev-Shwartz, Y.~Singer, N.~Srebro, and A.~Cotter, ``{Pegasos: Primal
  Estimated sub-GrAdient SOlver for SVM},'' \emph{Mathematical Programming,
  Series B}, vol. 127, no.~1, pp. 3--30, October 2011. [Online]. Available:
  \url{http://dx.doi.org/10.1007/s10107-010-0420-4}
\BIBentrySTDinterwordspacing

\bibitem{BianchiFortHachem:13IEEETrans}
\BIBentryALTinterwordspacing
P.~Bianchi, G.~Fort, and W.~Hachem, ``Performance of a distributed stochastic
  approximation algorithm,'' \emph{IEEE Transactions on Information Theory},
  vol.~59, no.~11, pp. 7405--7418, 2013. [Online]. Available:
  \url{http://dx.doi.org/10.1109/TIT.2013.2275131}
\BIBentrySTDinterwordspacing

\bibitem{dualAveraging}
\BIBentryALTinterwordspacing
J.~Duchi, A.~Agarwal, and M.~Wainwright, ``Dual averaging for distributed
  optimization: Convergence analysis and network scaling,'' \emph{IEEE
  Transactions on Automatic Control}, vol.~57, no.~3, pp. 592--606, March 2011.
  [Online]. Available: \url{http://dx.doi.org/10.1109/TAC.2011.2161027}
\BIBentrySTDinterwordspacing

\bibitem{distrStochSubgrOpt}
\BIBentryALTinterwordspacing
S.~S. Ram, A.~Nedic, and V.~V. Veeravalli, ``Distributed stochastic subgradient
  projection algorithms for convex optimization,'' \emph{Journal of
  Optimization Theory and Applications}, vol. 147, no.~3, pp. 516--545,
  December 2011. [Online]. Available:
  \url{http://dx.doi.org/10.1007/s10957-010-9737-7}
\BIBentrySTDinterwordspacing

\bibitem{LiuWrightReBittSri}
\BIBentryALTinterwordspacing
J.~Liu, S.~J. Wright, C.~Re, V.~Bittorf, and S.~Sridhar, ``An asynchronous
  parallel stochastic coordinate descent algorithm,'' in \emph{Proceedings of
  the 31st International Conference on Machine Learning}, ser. JMLR Workshop
  and Conference Proceedings, L.~Getoor and T.~Scheffer, Eds., vol.~32, 2014.
  [Online]. Available: \url{http://jmlr.org/proceedings/papers/v32/liud14.pdf}
\BIBentrySTDinterwordspacing

\bibitem{BradleyKyrola}
\BIBentryALTinterwordspacing
J.~K. Bradley, A.~Kyrola, D.~Bickson, and C.~Guestrin, ``Parallel coordinate
  descent for $l_1$-regularized loss minimization,'' in \emph{Proceedings of
  the 28th International Conference on Machine Learning}, ser. JMLR Workshop
  and Conference Proceedings, L.~Getoor and T.~Scheffer, Eds., vol.~28, 2011.
  [Online]. Available:
  \url{http://www.select.cs.cmu.edu/publications/paperdir/icml2011-bradley-kyr%
ola-bickson-guestrin.pdf}
\BIBentrySTDinterwordspacing

\bibitem{AgarwalD:11nips}
\BIBentryALTinterwordspacing
A.~Agarwal and J.~C. Duchi, ``Distributed delayed stochastic optimization,'' in
  \emph{Advances in Neural Information Processing Systems 24}, J.~Shawe-Taylor,
  R.~Zemel, P.~Bartlett, F.~Pereira, and K.~Weinberger, Eds., 2011, pp.
  873--881. [Online]. Available:
  \url{http://books.nips.cc/papers/files/nips24/NIPS2011_0574.pdf}
\BIBentrySTDinterwordspacing

\bibitem{TakacBRS:13icml}
\BIBentryALTinterwordspacing
M.~Tak\'{a}\v{c}, A.~Bijral, P.~Richt\'{a}rik, and N.~Srebro, ``Mini-batch
  primal and dual methods for {SVMs},'' in \emph{Proceedings of the 30th
  International Conference on Machine Learning (ICML)}, ser. JMLR Workshop and
  Conference Proceedings, S.~Dasgupta and D.~McAllester, Eds., vol.~28, 2013,
  pp. 1022--1030. [Online]. Available:
  \url{http://jmlr.org/proceedings/papers/v28/takac13.html}
\BIBentrySTDinterwordspacing

\bibitem{CotterSSS:11nips}
\BIBentryALTinterwordspacing
A.~Cotter, O.~Shamir, N.~Srebro, and K.~Sridharan, ``Better mini-batch
  algorithms via accelerated gradient methods,'' in \emph{Advances in Neural
  Information Processing Systems 24}, J.~Shawe-Taylor, R.~Zemel, P.~Bartlett,
  F.~Pereira, and K.~Weinberger, Eds., 2011, pp. 1647--1655. [Online].
  Available:
  \url{http://papers.nips.cc/paper/4432-better-mini-batch-algorithms-via-accel%
erated-gradient-methods}
\BIBentrySTDinterwordspacing

\bibitem{DekelGSX:12mini}
\BIBentryALTinterwordspacing
O.~Dekel, R.~Gilad-Bachrach, O.~Shamir, and L.~Xiao, ``Optimal distributed
  online prediction using mini-batches,'' \emph{Journal of Machine Learning
  Research}, vol.~13, pp. 165--202, January 2012. [Online]. Available:
  \url{http://jmlr.org/papers/v13/dekel12a.html}
\BIBentrySTDinterwordspacing

\bibitem{ShamirSrebroZhang:14icml}
\BIBentryALTinterwordspacing
O.~Shamir, N.~Srebro, and T.~Zhang, ``Communication-efficient distributed
  optimization using an approximate newton-type method,'' in \emph{Proceedings
  of the 31st International Conference on Machine Learning}, ser. JMLR Workshop
  and Conference Proceedings, E.~P. Xing and T.~Jebara, Eds., vol.~32, 2014,
  pp. 1000--1008. [Online]. Available:
  \url{http://jmlr.org/proceedings/papers/v32/shamir14.html}
\BIBentrySTDinterwordspacing

\bibitem{ZhangDW:12}
\BIBentryALTinterwordspacing
Y.~Zhang, J.~Duchi, and M.~Wainwright, ``Communication-efficient algorithms for
  statistical optimization,'' in \emph{Advances in Neural Information
  Processing Systems 25}, P.~Bartlett, F.~Pereira, C.~Burges, L.~Bottou, and
  K.~Weinberger, Eds., 2012, pp. 1511--1519. [Online]. Available:
  \url{http://books.nips.cc/papers/files/nips25/NIPS2012_0716.pdf}
\BIBentrySTDinterwordspacing

\bibitem{StochConOpt}
S.~Shalev-Shwartz, O.~Shamir, N.~Srebro, and K.~Sridharan, ``Stochastic convex
  optimization,'' in \emph{Proceedings Conference on Learning Theory},
  P.~Bartlett, F.~Pereira, C.~Burges, L.~Bottou, and K.~Weinberger, Eds., 2009.

\bibitem{MokhtariR:16jmlr}
\BIBentryALTinterwordspacing
A.~Mokhtari and A.~Ribeiro, ``{DSA:} decentralized double stochastic averaging
  gradient algorithm,'' \emph{Journal of Machine Learning Research}, vol.~17,
  no.~61, pp. 1--35, 2016. [Online]. Available:
  \url{http://jmlr.org/papers/v17/15-292.html}
\BIBentrySTDinterwordspacing

\bibitem{ShiLWY:15extra}
W.~Shi, Q.~Ling, G.~Wu, and W.~Yin, ``{EXTRA}: An exact first-order algorithm
  for decentralized consensus optimization,'' \emph{SIAM Journal on
  Optimization}, vol.~25, no.~2, pp. 944--966, 2015.

\bibitem{SchmidLF:2015}
\BIBentryALTinterwordspacing
M.~Schmidt, N.~L. Roux, and F.~Bach, ``Minimizing finite sums with the
  stochastic average gradient,'' HAL, Tech. Rep. hal-00860051, January 2015.
  [Online]. Available: \url{https://hal.inria.fr/hal-00860051v2}
\BIBentrySTDinterwordspacing

\bibitem{DistStronglyConvex}
\BIBentryALTinterwordspacing
K.~I. Tsianos and M.~G. Rabbat, ``Distributed strongly convex optimization,''
  ArXiV, Tech. Rep. arXiv:1207.3031 [cs.DC], July 2012. [Online]. Available:
  \url{http://arxiv.org/abs/1207.3031}
\BIBentrySTDinterwordspacing

\bibitem{TsianosNIPS2012}
K.~I. Tsianos, S.~Lawlor, and M.~G. Rabbat, ``Communication/computation
  tradeoffs in consensus-based distributed optimization,'' in \emph{Advances in
  Neural Information Processing Systems 25}, P.~Bartlett, F.~Pereira,
  C.~Burges, L.~Bottou, and K.~Weinberger, Eds., 2012, pp. 1952--1960.

\bibitem{covertype}
\BIBentryALTinterwordspacing
J.~A. Blackard and D.~J. Dean, ``Comparative accuracies of artificial neural
  networks and discriminant analysis in predicting forest cover types from
  cartographic variables,'' \emph{Computers and Electronics in Agriculture},
  vol.~24, no.~3, pp. 131--151, December 1999. [Online]. Available:
  \url{http://dx.doi.org/10.1016/S0168-1699(99)00046-0}
\BIBentrySTDinterwordspacing

\bibitem{Lichman:2013}
\BIBentryALTinterwordspacing
M.~Lichman, ``{UCI} machine learning repository,'' 2013. [Online]. Available:
  \url{http://archive.ics.uci.edu/ml}
\BIBentrySTDinterwordspacing

\bibitem{ChangL:11libsvm}
\BIBentryALTinterwordspacing
C.-C. Chang and C.-J. Lin, ``Libsvm: A library for support vector machines,''
  \emph{ACM Trans. Intell. Syst. Technol.}, vol.~2, no.~3, pp. 27:1--27:27, May
  2011. [Online]. Available: \url{http://dx.doi.org/10.1145/1961189.1961199}
\BIBentrySTDinterwordspacing

\bibitem{DiaconisBoyd}
\BIBentryALTinterwordspacing
S.~Boyd, L.~Xiao, and P.~Diaconis, ``Fastest mixing markov chain on a graph,''
  \emph{SIAM Review}, vol.~46, no.~4, pp. 667--689, 2004. [Online]. Available:
  \url{http://dx.doi.org/10.1137/S0036144503423264}
\BIBentrySTDinterwordspacing

\bibitem{DoubStoch::nips09}
A.~S. Bijral and N.~Srebro, ``On doubly stochastic graph optimization,'' in
  \emph{NIPS Workshop on Analyzing Networks and Learning with Graphs}, 2009.

\bibitem{Tropp:12tail}
\BIBentryALTinterwordspacing
J.~A. Tropp, ``User-friendly tail bounds for sums of random matrices,''
  \emph{Foundations of Computational Mathematics}, vol.~12, no.~4, pp.
  389--434, August 2012. [Online]. Available:
  \url{http://dx.doi.org/10.1007/s10208-011-9099-z}
\BIBentrySTDinterwordspacing

\end{thebibliography}

\end{document}